\documentclass[final,3p]{elsarticle}
 \usepackage{graphics}
 \usepackage{graphicx}
 \usepackage{epsfig}
\usepackage{amssymb}
 \usepackage{amsthm}
 \usepackage{lineno}
 \usepackage{amsmath}
   \numberwithin{equation}{section}
\usepackage{mathrsfs}

\newtheorem{thm}{Theorem}[section]
\newtheorem{cor}[thm]{Corollary}
\newtheorem{lem}[thm]{Lemma}

\newtheorem{defn}[thm]{Definition}

\biboptions{sort&compress,square}
\begin{document}
\begin{frontmatter}
\author[rvt1,rvt2]{Jin Hong}
\ead{jhong@nenu.edu.cn}
\author[rvt1]{Yuchen Yang}
\ead{yangyc580@nenu.edu.cn}
\author[rvt1]{Yong Wang\corref{cor2}}
\ead{wangy581@nenu.edu.cn}
\cortext[cor2]{Corresponding author.}
\address[rvt1]{School of Mathematics and Statistics, Northeast Normal University, Changchun, 130024, China}
\address[rvt2]{School of Mathematics and Statistics, Yili Normal University, Yining 835000, China}

\title{ Sub-Dirac operators, spectral Einstein functionals\\ and the noncommutative residue}
\begin{abstract}
In this paper, we define the spectral Einstein functional associated with the sub-Dirac operator for manifolds with boundary. A proof of the Dabrowski-Sitarz-Zalecki type theorem for spectral Einstein functions associated with the sub-Dirac operator on four-dimensional manifolds with boundary is also given.
\end{abstract}
\end{frontmatter}
\section{Introduction}
\label{1}
The noncommutative residue which is found in \cite{Gu,Wo1} by M. Wodzicki and V. W. Guillemin plays a prominent role in noncommutative geometry.
Recently Dabrowski etc. \cite{DL} obtained the metric and Einstein functionals by two vector fields and Laplace-type operators over vector bundles, giving an interesting example of the spinor connection and square of the Dirac operator.
An eminent spectral scheme is the small-time asymptotic expansion of the (localised) trace of heat kernel \cite{PBG,FGV} 
that generates geometric objects on manifolds such as residue, scalar curvature, and other scalar combinations of curvature tensors. 
The theory has very rich structures both in physics and mathematics.

In \cite{Co1}, Connes used the noncommutative residue to derive a conformal 4-dimensional Polyakov action analogy. Connes proved that the noncommutative residue on a compact manifold M coincided with Dixmier’s trace on pseudodifferential operators of order $-dimM$ \cite{Co2}. Furthermore, Connes claimed the noncommutative residue of the square of the inverse of the Dirac operator was proportioned to the Einstein-Hilbert action in\cite{Co2,Co3}. Kastler provided a brute-force proof of this theorem in \cite{Ka}, while Kalau and Walze proved it in the normal coordinates system simultaneously in \cite{KW}, which is called the Kastler-Kalau-Walze theorem now. Building upon the theory of the noncommutative reside introduced by Wodzicki, Fedosov etc. \cite{FGLS} constructed a noncommutative residue on the algebra of classical elements in Boutet de Monvel’s calculus on a compact manifold with boundary of dimension $n>2$. With elliptic pseudodifferential operators and noncommutative residue, it is natural way for investigating the Kastler-Kalau-Walze type theorem and operator-theoretic explanation of the gravitational action on manifolds with boundary. Concerning Dirac operators and signature operators, Wang performed computations of the noncommutative residue and successfully demonstrated the Kastler-Kalau-Walze type theorem for manifolds with boundaries \cite{Wa1,Wa3,Wa4}.

 Earlier Jean-Michel Bismut \cite{JMB} proved a local index theorem for Dirac operators on a
Riemannian manifold $M$ associated with connections on $TM$ which have non zero
torsion.
In \cite{AT}, Ackermann and Tolksdorf proved a generalized version of the well-known Lichnerowicz formula for the square of the
most general Dirac operator with torsion $D_{T}$ on an even-dimensional spin manifold associated to a metric connection with torsion.
 Pf$\ddot{a}$ffle and Stephan considered compact Riemannian spin manifolds without boundary equipped with orthogonal connections,
and investigated the induced Dirac operators in \cite{PS}. 
Wang, Wang and Wu computed $\widetilde{wres}[\pi^+\tilde\nabla_{X}\tilde\nabla_{Y}(D_T^*D_T)^{-1}\circ\pi^+(D^*_TD_T)^{-1}]$ and
$\widetilde{wres}[\pi^+\tilde\nabla_{X}\tilde\nabla_{Y}D_T^{-1}\circ\pi^+(D^*_TD_TD^*_T)]$ in \cite{WWw}, and  provided an important reference for our paper calculation. In \cite{LZ}, in order to prove the Connes vanshing theorem, Liu and Zhang introduced the sub-Dirac operator. In \cite{WW} Wang and Wang defined lower dimensional volumes associated to sub-Dirac operators for foliations and compute these lower dimensional volumes. They also prove the Kastler-Kalau-Walze type theorems for foliations with or without boundary.
 
The purpose of this paper is to generalize the results in \cite{DL}, \cite{WW}, \cite{WWw}, and get spectral functionals
 associated with sub-Dirac operators on compact manifolds with  boundary. For lower dimensional compact Riemannian manifolds
  with  boundary, we compute the 4-dimensional  residue of $\widetilde{\nabla}_{X}\widetilde{\nabla}_{Y}D_{F}^{-4}$ and
   get Dabrowski-Sitarz-Zalecki theorems.

\section{Spectral functionals for the sub-Dirac operator }

 In this section, we shall restrict our attention to the sub-Dirac operators for foliations. Let $ (M, F) $ be a closed foliation and  $M $ has spin leave, $ g^{F} $ be a metric on $ F $. Let $ g^{T M} $ be a metric on $ T M $ which restricted to  $g^{F}$  on $ F $. Let $ F^{\perp} $ be the orthogonal complement of $ F $ in $ T M $ with respect to $ g^{T M} $. Then we have the following orthogonal splitting
 
 \begin{align}
 	&T M=F \oplus F^{\perp}, \\
 	&g^{T M}=g^{F} \oplus g^{F^{\perp}},
 \end{align}
 where  $g^{F^{\perp}} $ is the restriction of  $g^{T M}$  to  $F^{\perp}$ .
 
 Let  $P, P^{\perp} $ be the orthogonal projection from  $T M$  to  $F, F^{\perp} $ respectively. Let  $\nabla^{T M} $ be the Levi-Civita connection of $ g^{T M} $ and  $\nabla^{F}  (resp.  \nabla^{F^{\perp}}  ) $be the restriction of $ \nabla^{T M} $ to $ F  (resp.  F^{\perp}  )$. Without loss of generality, we assume $ F$  is oriented, spin and carries a fixed spin structure. Furthermore, we assume  $F^{\perp}$  is oriented and we do not assume that  $\operatorname{dim} F$  and  $\operatorname{dim} F^{\perp}$  are even. By assumption, we may write
 \begin{align}
 	&\nabla^{F}=P \nabla^{T M} P ,\\
 	&\nabla^{F^{\perp}}=P^{\perp} \nabla^{T M} P^{\perp} .
 \end{align}
 
 Moreover, the connections  $\nabla^{F}\left(\nabla^{F^{\perp}}\right)$  lift to  $S(F)\left(\wedge\left(F^{\perp, \star}\right)\right)$  naturally denoted by  $\nabla^{S(F)}\left(\nabla^{\wedge\left(F^{\perp, \star}\right)}\right) $. Then  $S(F) \otimes \wedge\left(F^{\perp, \star}\right)$  carries the induced tensor product connection
 \begin{align}
 \nabla^{S(F) \otimes \wedge\left(F^{\perp, \star}\right)}=\nabla^{S(F)} \otimes \operatorname{Id}_{\wedge\left(F^{\perp, \star}\right)}+\operatorname{Id}_{S(F)} \otimes \nabla^{\wedge\left(F^{\perp, \star}\right)} .
 \end{align} 
 
 Then we can define $ S \in \Omega\left(T^{*} M\right) \otimes \Gamma(\operatorname{End}(T M)) $,
  \begin{align}
 \nabla^{T M}=\nabla^{F}+\nabla^{F^{\perp}}+S
  \end{align}
 
 For any $ X \in \Gamma(T M), S(X)$  exchanges $ \Gamma(F)$  and  $\Gamma\left(F^{\perp}\right)$  and is skew-adjoint with respect to  $g^{T M}$ . Let $ \left\{f_{i}\right\}_{i=1}^{p} $ be an oriented orthonormal basis of  $F$ , $ \left\{h_{s}\right\}_{s=1}^{q} $ be an oriented orthonormal basis of  $F^{\perp}$ , we define
  \begin{align}
 \widetilde{\nabla}^{F}=\nabla^{S(F) \otimes \wedge\left(F^{\perp, \star}\right)}+\frac{1}{2} \sum_{j=1}^{p} \sum_{s=1}^{q}<S(.) f_{j}, h_{s}>c\left(f_{j}\right) c\left(h_{s}\right)
  \end{align}
where the vector bundle $F^{\perp}$ might well be non-spin.

\begin{defn}\cite{WW}
 Let $ D_{F}$  be the operator mapping from  $\Gamma\left(S(F) \otimes \wedge\left(F^{\perp, \star}\right)\right)$  to itself defined by
\begin{align}
D_{F}=\sum_{i=1}^{p} c\left(f_{i}\right) \widetilde{\nabla}_{f_{i}}^{F}+\sum_{s=1}^{q} c\left(h_{s}\right) \widetilde{\nabla}_{h_{s}}^{F} .
\end{align}
From (2.19) in \cite{LZ}, we shall make use of the Bochner Laplacian  $\triangle^{F} $ stating that
\begin{align}
\triangle^{F}:=-\sum_{i=1}^{p}\left(\widetilde{\nabla}_{f_{i}}\right)^{2}-\sum_{s=1}^{q}\left(\widetilde{\nabla}_{h_{s}}\right)^{2}+\widetilde{\nabla}_{\sum_{i=1}^{p}} \nabla_{f_{i}}^{T M} f_{i}+\widetilde{\nabla}_{\sum_{s=1}^{q} \nabla_{h_{s}}^{T M} h_{s}} .
\end{align}
\end{defn}
Let  $r_{M}$  be the scalar curvature of the metric  $g^{T M}$ . Let  $R^{F^{\perp}}$  be the curvature tensor of  $F^{\perp}$ . From theorem 2.3 in \cite{LZ}, we have the following Lichnerowicz formula for  $D_{F} $.

\begin{lem}\cite{LZ} The following identity holds
\begin{align}
	D_{F}^{2}= & \triangle^{F}+\frac{r_{M}}{4}+\frac{1}{4} \sum_{i=1}^{p} \sum_{r, s, t=1}^{q}\left\langle R^{F^{\perp}}\left(f_{i}, h_{r}\right) h_{t}, h_{s}\right\rangle c\left(f_{i}\right) c\left(h_{r}\right) \widehat{c}\left(h_{s}\right) \widehat{c}\left(h_{t}\right) \nonumber\\
	& +\frac{1}{8} \sum_{i, j=1}^{p} \sum_{s, t=1}^{q}\left\langle R^{F^{\perp}}\left(f_{i}, f_{j}\right) h_{t}, h_{s}\right\rangle c\left(f_{i}\right) c\left(f_{j}\right) \widehat{c}\left(h_{s}\right) \widehat{c}\left(h_{t}\right) \nonumber\\
	& +\frac{1}{8} \sum_{s, t, r, u=1}^{q}\left\langle R^{F^{\perp}}\left(h_{r}, h_{l}\right) h_{t}, h_{s}\right\rangle c\left(h_{r}\right) c\left(h_{u}\right) \widehat{c}\left(h_{s}\right) \widehat{c}\left(h_{t}\right) .
\end{align}
\end{lem}

Let us now turn to compute the specification of $D_{F}^3 $.
\begin{align}
D_{F}^{3}=&-\sum_{l} g^{t l}(x) c\left(\partial_{t}\right) g^{i j} \partial_{i} \partial_{j} \partial_{l}+\left[\sum_{l} g^{t l}(x) c\left(\partial_{t}\right)\right]\left\{\sum_{i j}\left(\partial_{l} g^{i j}\right) \partial_{i} \partial_{j}-4 \sum_{i j} g^{i j} \sigma_{i} \partial_{l} \partial_{j} \otimes I d_{\wedge\left(F^{\perp, \star}\right)}\right. \nonumber\\
&-I d_{S(F)} \otimes 4 \sum_{i j} g^{i j} \tilde{\sigma}_{i} \partial_{l} \partial_{j}-4 \sum_{i j} g^{i j} s\left(\partial_{j}\right) \partial_{l} \partial_{i}+2 \sum_{i j} g^{i j} \sum_{k l} \Gamma_{i j}^{k} \partial_{l} \partial_{k} \nonumber\\
&\left.-\left[\frac{1}{4} \sum_{s, t=1}^{p}\left\langle\nabla_{\partial_{j}}^{F} f_{s}, f_{t}\right\rangle c\left(f_{s}\right) c\left(f_{t}\right)+\frac{1}{4} \sum_{l, \alpha=1}^{q}\left\langle\nabla_{\partial_{j}}^{F^{\perp}} h_{l}, h_{\alpha}\right\rangle\left[c\left(h_{l}\right) c\left(h_{\alpha}\right)-\hat{c}\left(h_{l}\right) \hat{c}\left(h_{\alpha}\right)\right]\right] \sum_{i, j} g^{i j} \partial_{i} \partial_{j}\right\} \nonumber\\
&+\left[\sum_{l} g^{t l}(x) c\left(\partial_{t}\right)\right]\left\{-2 \sum_{i j}\left(\partial_{l} g^{i j}\right) \sigma_{i} \partial_{j} \otimes I d_{\wedge\left(F^{\perp, \star}\right)}-2 \sum_{i j} g^{i j}\left(\partial_{l} \sigma_{i}\right) \partial_{j} \otimes I d_{\wedge\left(F^{\perp, \star)}\right.}\right. \nonumber\\
&-I d_{S(F) \otimes 2}\left(\partial_{l} g^{i j}\right) \tilde{\sigma}_{i} \partial_{j}-I d_{S(F)} \otimes 2 \sum_{i j} g^{i j}\left(\partial_{l} \tilde{\sigma}_{i}\right) \partial_{j}-2 \sum_{i j}\left(\partial_{l} g^{i j}\right) S\left(\partial_{j}\right) \partial_{i} \nonumber\\
&\left.-2 \sum_{i j} g^{i j}\left[\partial_{l} S\left(\partial_{j}\right)\right] \partial_{i}+\sum_{i j}\left(\partial_{l} g^{i j}\right) \sum_{k} \Gamma_{i j}^{k} \partial_{k}+\sum_{i j} g^{i j} \sum_{k}\left(\partial_{l} \Gamma_{i j}^{k}\right) \partial_{k}\right\} \nonumber\\
&+\bigg\{\frac{1}{4} \sum_{s, t=1}^{p}\left\langle\nabla_{\partial_{j}}^{F} f_{s}, f_{t}\right\rangle c\left(f_{s}\right) c\left(f_{t}\right)+\frac{1}{4} \sum_{l, \alpha=1}^{q}\left\langle\nabla_{\partial_{j}}^{F^{\perp}} h_{l}, h_{\alpha}\right\rangle\left[c\left(h_{l}\right) c\left(h_{\alpha}\right)-\hat{c}\left(h_{l}\right) \hat{c}\left(h_{\alpha}\right)\right]\bigg\} \nonumber\\
&\times\left\{-\sum_{i j} g^{i j}\left[2 \sigma_{i} \partial_{j} \otimes I d_{\wedge\left(F^{\perp, *}\right)}+I d_{S(F)} \otimes 2 \tilde{\sigma}_{i} \partial_{j}+S\left(\partial_{j}\right) \partial_{i}+S\left(\partial_{i}\right) \partial_{j}-\sum_{k} \Gamma_{i j}^{k} \partial_{k}\right]+\frac{r_{M}}{4}\right. \nonumber\\
&+\frac{1}{4} \sum_{i=1}^{p} \sum_{r, s, t=1}^{q}\left\langle R^{F^{\perp}}\left(f_{i}, h_{r}\right) h_{t}, h_{s}\right\rangle c\left(f_{i}\right) c\left(h_{r}\right) \widehat{c}\left(h_{s}\right) \widehat{c}\left(h_{t}\right) \nonumber\\
&+\frac{1}{8} \sum_{i, j=1}^{p} \sum_{s, t=1}^{q}\left\langle R^{F^{\perp}}\left(f_{i}, f_{j}\right) h_{t}, h_{s}\right\rangle c\left(f_{i}\right) c\left(f_{j}\right) \widehat{c}\left(h_{s}\right) \widehat{c}\left(h_{t}\right) \nonumber\\
&\left.+\frac{1}{8} \sum_{s, t, r, u=1}^{q}\left\langle R^{F^{\perp}}\left(h_{r}, h_{l}\right) h_{t}, h_{s}\right\rangle c\left(h_{r}\right) c\left(h_{u}\right) \widehat{c}\left(h_{s}\right) \widehat{c}\left(h_{t}\right)\right\} . 
\end{align}

In order to get a Kastler-Kalau-Walze type theorem for foliations, Liu and Wang \cite{LW} considered the noncommutative residue of the $ -n+2 $ power of the sub-Dirac operator, and got the following Kastler-Kalau-Walze type theorem for foliations.

The following lemma of Dabrowski etc.'s Einstein functional play a key role in our proof
of the Einstein functional .
Let $V$, $W$ be a pair of vector fields on a compact Riemannian manifold $M$, of dimension $n = 2m$. Using the Laplace operator $\Delta=-(\sum_{j=1}^{n}\widetilde{\nabla}_{e_j}\widetilde{\nabla}_{e_j}-\widetilde{\nabla}_{\nabla_{e_j}^L e_j})+E $
acting on sections of a vector bundle $\overline{E}$ where $\widetilde{\nabla}$ is a connetion on $\overline{E}$, which
 may contain both some nontrivial connections and torsion,
 the spectral functionals over vector fields defined by
\begin{lem}\cite{DL}
The Einstein functional equal to
 \begin{equation}
Wres\big(\widetilde{\nabla}_{V}\widetilde{\nabla}_{W}\Delta^{-m}\big)=\frac{\upsilon_{n-1}}{6}2^{m}\int_{M}G(V,W)vol_{g}
 +\frac{\upsilon_{n-1}}{2}\int_{M}F(V,W)vol_{g}+\frac{1}{2}\int_{M}(\mathrm{Tr}E)g(V,W)vol_{g},
\end{equation}
where $G(V,W)$ denotes the Einstein tensor evaluated on the two vector fields, $F(V,W)=Tr(V_{a}W_{b}F_{ab})$ and
$F_{ab}$ is the curvature tensor of the connection $T$, $\mathrm{Tr}E$ denotes the trace of $E$ and $\upsilon_{n-1}=\frac{2\pi^{m}}{\Gamma(m)}$.
\end{lem}
The aim of this section is to prove the following.
\begin{thm}
For the Laplace (type) operator $\Delta_{F}=D_{F}^2$, the Einstein functional equal to
\begin{align}
Wres\big(\widetilde{\nabla}^{F}_{V}\widetilde{\nabla}^{F}_{W}({D}^{F})^{-2m}\big)
=&\frac{2^{\frac{p}{2}+q+1}\pi^\frac{n}{2}}{6\Gamma(\frac{p+q}{2})}\int_{M}G(V,W) dvol_{g}+2^{\frac{p}{2}+q-3}\int_{M}sg(V, W)dvol_{M},
\end{align}
where $s$ is the scalar curvature.
\end{thm}

\begin{proof}
	By the definition of connection $\widetilde{\nabla}^{F}$, we have
 \begin{align}
\widetilde{\nabla}_{X}^{F}=&{\nabla}_{X}^{S(F)\otimes(F^\bot,*)}+\frac{1}{2}\sum_{j=1}^{p}\sum_{s=1}^{q}\langle S(X)f_j,h_s\rangle c(f_j)c(h_s)\nonumber\\
=&X+\frac{1}{4}\sum_{j,l=1}^{p}\langle {\nabla}^{F}_{X}f_j,f_l\rangle c(f_j)c(f_l)+\frac{1}{4}\sum_{s,t=1}^{q}\langle {\nabla}^{F^\bot}_{X}h_s,h_t\rangle (c(h_s)c(h_t)-\hat{c}(h_s)\hat{c}(h_t))\nonumber\\
&+\frac{1}{2}\sum_{j=1}^{p}\sum_{s=1}^{q}\langle S(X)f_j,h_s\rangle c(f_j)c(h_s)\nonumber\\
:=&X+\overline{A}(X).
\end{align}

Let $V=\sum_{a=1}^{n}V^{a}e_{a}$, $W=\sum_{b=1}^{n}W^{b}e_{b}$,
in view of that
 \begin{equation}
F(V,W)=Tr(V_{a}W_{b}F_{ab})=\sum_{a,b=1}^{n}V_{a}W_{b}Tr^{S(F)\otimes(F^\bot,*)}(F_{e_{a},e_{b}}),
\end{equation}
we obtain
 \begin{align}
F_{e_{a},e_{b}}=e_{a}(\overline{A}(e_{b}))-e_{b}(\overline{A}(e_{a}))+\overline{A}(e_{a})\overline{A}(e_{b})-\overline{A}(e_{b})\overline{A}(e_{a})
-\overline{A}([e_{a},e_{b}]).
\end{align}
Also, straightforward computations yield
 \begin{align}
\mathrm{Tr}\big(e_{a}(\overline{A}(e_{b}))\big)
=&\mathrm{Tr}\Big[e_{a}(\frac{1}{4}\sum_{j,l=1}^{p}\langle{\nabla}^{F}_{X}f_j,f_l\rangle c(f_j)c(f_l)+\frac{1}{4}\sum_{s,t=1}^{q}\langle{\nabla}^{F^\bot}_{X}h_s,h_t\rangle (c(h_s)c(h_t)-\hat{c}(h_s)\hat{c}(h_t))\nonumber\\
&+\frac{1}{2}\sum_{j=1}^{p}\sum_{s=1}^{q}\langle S(X)f_j,h_s\rangle c(f_j)c(h_s))]\nonumber\\
=0,
\end{align}
in the same vein
 \begin{align}
	\mathrm{Tr}\overline{A}([e_{a},e_{b}])=0,
\end{align}
so
\begin{align}
	F(V,W)=0.
\end{align}
Let $D_F^2=\Delta^F+E$, we have
 \begin{align}
E=&\frac{s}{4}+\frac{1}{4}\sum_{i=1}^{p}\sum_{r,s,t=1}^{q}\langle R^{F^\bot}(f_i,h_r)h_t,h_s\rangle c(f_i)c(h_r)\hat{c}(h_s)\hat{c}(h_t)\nonumber\\
&+\frac{1}{4}\sum_{i,j=1}^{p}\sum_{s,t=1}^{q}\langle R^{F^\bot}(f_i,f_j)h_t,h_s\rangle c(f_i)c(f_j)\hat{c}(h_s)\hat{c}(h_t)\nonumber\\
&+\frac{1}{4}\sum_{s,t,r,u=1}^{q}\langle R^{F^\bot}(h_r,h_u)h_t,h_s\rangle c(h_r)c(h_u)\hat{c}(h_s)\hat{c}(h_t),\nonumber\\
\end{align}
and
\begin{align}
    \mathrm{Tr}\langle R^{F^\bot}(f_i,h_r)h_t,h_s\rangle c(f_i)c(h_r)\hat{c}(h_s)\hat{c}(h_t)=0,
\end{align}
\begin{align}
	\mathrm{Tr}\langle R^{F^\bot}(f_i,f_j)h_t,h_s\rangle c(f_i)c(f_j)\hat{c}(h_s)\hat{c}(h_t)=0,
\end{align}
\begin{align}
	\mathrm{Tr}\langle R^{F^\bot}(h_r,h_u)h_t,h_s\rangle c(h_r)c(h_u)\hat{c}(h_s)\hat{c}(h_t)=0.
\end{align}
Hence
\begin{align}
	\mathrm{Tr}E=\frac{s}{4}\mathrm{Tr}[Id]=\frac{s}{4}{2}^{\frac{p}{2}+q}={2}^{\frac{p}{2}+q-2}s,
\end{align}
which finishes the proof of Theorem 2.4.
\end{proof}

 \section{The residue for the sub-Dirac operator $\widetilde{\nabla}^{F}_{X}\widetilde{\nabla}^{F}_{Y}D^{-2}_{F}$ and $D^{-2}_{F}$ } 
In this section, we compute the lower dimensional volume for 4-dimension compact manifolds with boundary and get a
Dabrowski-Sitarz-Zalecki type formula in this case.
Some basic knowledge such as boundary metric, frame, noncommutative residue of manifold with boundary and Boutet de Monvel’s algebra can be referred to in \cite{Wa1}, and we will not repeat it here.
 We will consider $D^{2}_{F}$.
Since $[\sigma_{-4}(\widetilde{\nabla}^F_{X}\widetilde{\nabla}^F_{Y}D^{-2}_F\circ D^{-2}_{F})]|_{M}$ has
the same expression as $[\sigma_{-4}(\widetilde{\nabla}^{F}_{X}\widetilde{\nabla}^{F}_{Y}D^{-2}_{F}
  \circ D^{-2}_F)]|_{M}$ in the case of manifolds without boundary,
so locally we can use Theorem 2.4 to compute the first term.
\begin{cor}
 Let M be a 4-dimensional compact manifold without boundary and $\widetilde{\nabla}^{F}$ be an orthogonal
connection. Then we get the volumes  associated to $\widetilde{\nabla}_{X}^{F}\widetilde{\nabla}_{Y}^{F}D_{F}^{-2}$
and $D_{F}^2$ on compact manifolds without boundary
 \begin{align}
Wres[\sigma_{-4}(\widetilde{\nabla}_{X}^{F}\widetilde{\nabla}_{Y}^{F}D_{F}^{-2}
  \circ D_{F}^{-2})]=\frac{8\pi^2}{3}\int_{M}G(X,Y) dvol_{g}+\int_{M}sg(X, Y)dvol_{M},
\end{align}
where  $s$ is the scalar curvature.
\end{cor}

Let $p_{1},p_{2}$ be nonnegative integers and $p_{1}+p_{2}\leq n$,
denote by $\sigma_{l}(\widetilde{A})$ the $l$-order symbol of an operator $\widetilde{A}$,
an application of (3.5) and (3.6) in \cite{Wa1} shows that
\begin{defn} Spectral Einstein functional associated the sub-Dirac operator of manifolds with boundary is defined by
	\begin{equation}\label{}
		vol_{n}^{\{p_{1},p_{2}\}}M:=\widetilde{Wres}[\pi^{+}(\widetilde{\nabla}_{X}^{F}\widetilde{\nabla}_{Y}^{F}(D_{F}^2)^{-p_{1}})
		\circ\pi^{+}(D_{F}^{-2})^{p_{2}}],
	\end{equation}
	where $\pi^{+}(\widetilde{\nabla}_{X}^{F}\widetilde{\nabla}_{Y}^{F}(D_{F}^2)^{-p_{1}})$, $\pi^{+}(D_{F}^{-2})^{p_{2}}$ are
	elements in Boutet de Monvel's algebra \cite{Wa3}.
\end{defn}
For the sub-Dirac operator
$\widetilde{\nabla}_{X}^{F}\widetilde{\nabla}_{Y}^{F}D_{F}^{-2}$ and $D_{F}^{-2}$,
denote by $\sigma_{l}(\widetilde{A})$ the $l$-order symbol of an operator $\widetilde{A}$. An application of (2.1.4) in \cite{Wa1} shows that
\begin{align}
	&\widetilde{Wres}[\pi^{+}(\widetilde{\nabla}_{X}^{F}\widetilde{\nabla}_{Y}^{F}(D_{F}^{-2})^{p_{1}})
	\circ\pi^{+}(D_{F}^{2})^{-p_{2}}]\nonumber\\
	&=\int_{M}\int_{|\xi|=1}\mathrm{tr}_{S(F)\otimes\wedge(F^\bot,*)}
	[\sigma_{-n}(\widetilde{\nabla}_{X}^{F}\widetilde{\nabla}_{Y}^{F}(D_{F}^2)^{-p_{1}}
	\circ (D_{F}^{2})^{-p_{2}}]\sigma(\xi)\texttt{d}x+\int_{\partial M}\Phi,
\end{align}
where
\begin{align}\label{c4}
	\Phi=&\int_{|\xi'|=1}\int_{-\infty}^{+\infty}\sum_{j,k=0}^{\infty}\sum \frac{(-i)^{|\alpha|+j+k+l}}{\alpha!(j+k+1)!}
	\mathrm{tr}_{S(F)\otimes\wedge(F^\bot,*)}[\partial_{x_{n}}^{j}\partial_{\xi'}^{\alpha}\partial_{\xi_{n}}^{k}\sigma_{r}^{+}
	(\widetilde{\nabla}_{X}^{F}\widetilde{\nabla}_{Y}^{F}(D_{F}^{2})^{-p_{1}})(x',0,\xi',\xi_{n})\nonumber\\
	&\times\partial_{x_{n}}^{\alpha}\partial_{\xi_{n}}^{j+1}\partial_{x_{n}}^{k}\sigma_{l}((D_{F}^2)^{-p_{2}})(x',0,\xi',\xi_{n})]
	\texttt{d}\xi_{n}\sigma(\xi')\texttt{d}x' ,
\end{align}
and the sum is taken over $r-k-|\alpha|+\ell-j-1=-n,r\leq-p_{1},\ell\leq-p_{2}$.\\

So we only need to compute $\int_{\partial M}\Phi$.
 Recall the definition of the sub-Dirac operator $D_{F}$ in \cite{WW}. 
\begin{equation}
D_F=\sum_{i=1}^{p}c({f_{i}})\widetilde{\nabla}_{f_i}+\sum_{s=1}^{q}c({h_{s}})\widetilde{\nabla}_{h_s},
\end{equation}
where $c({f_{i}}),c({h_{s}})$ denotes the Clifford action. 

We define
\begin{align}
\widetilde\nabla_X^F:=&X+\frac{1}{2}\sum_{j=1}^{p}\sum_{s=1}^{q}\langle S(X)f_j,h_s\rangle c(f_j)c(h_s)+\frac{1}{4}\sum_{j,l=1}^{p}\langle \nabla_{X}^{F}f_j,f_l\rangle c(f_j)c(f_l)\nonumber\\
&+\frac{1}{4}\sum_{s,t=1}^{p}\langle \nabla_{X}^{F^\bot}h_s,h_t\rangle [c(h_s)c(h_t)-\hat{c}(h_s)\hat{c}(h_t)].
\end{align}
 which is a connection
 on ${S(F)\otimes\wedge(F^\bot,*)}$.
 
 Set
\begin{align}
&A(X)=\frac{1}{2}\sum_{j=1}^{p}\sum_{s=1}^{q}\langle S(X)f_j,h_s\rangle c(f_j)c(h_s);
~~M(X)=\frac{1}{4}\sum_{j,l=1}^{p}\langle \nabla_{X}^{F}f_j,f_l\rangle c(f_j)c(f_l);\nonumber\\
&N(X)=\frac{1}{4}\sum_{s,t=1}^{p}\langle \nabla_{X}^{F^\bot}h_s,h_t\rangle [c(h_s)c(h_t)-\hat{c}(h_s)\hat{c}(h_t)].
\end{align}
 
 Let $\widetilde{\nabla}_{X}^{F}=X+A(X)+M(X)+N(X)$, and
 $\widetilde{\nabla}_{Y}^{F}=Y+A(Y)+M(Y)+N(Y)$, by (2.11), we obtain
\begin{align}
\widetilde{\nabla}^{F}_{X}\widetilde{\nabla}^{F}_{Y}&=(X+A(X)+M(X)+N(X))(Y+A(Y)+M(Y)+N(Y))\nonumber\\
 &=XY+M(Y)X+N(Y)X+A(Y)X+M(X)Y+N(X)Y+A(X)Y\nonumber\\
 &~~~~+X[M(Y)]+X[N(Y)]+X[A(Y)]+M(X)M(Y)+M(X)N(Y)+M(X)A(Y)\nonumber\\
 &~~~~+N(X)M(Y)+N(X)N(Y)+N(X)A(Y)+A(X)M(Y)+A(X)N(Y)+A(X)A(Y),
\end{align}
where
$X=\Sigma_{j=1}^nX_j\partial_{x_j}, Y=\Sigma_{l=1}^nY_l\partial_{x_l}$.

Let $g^{ij}=g(dx_{i},dx_{j})$, $\xi=\sum_{k}\xi_{j}dx_{j}$
  and $\nabla^L_{\partial_{i}}\partial_{j}=\sum_{k}\Gamma_{ij}^{k}\partial_{k}$,
   we get
\begin{align}
&\sigma_{k}=-\frac{1}{4} \sum_{k, l} \omega_{k, l}\left(\partial_{k}\right) c\left(f_{k}\right) c\left(f_{l}\right);~~~\tilde{\sigma}_{k}=\frac{1}{4} \sum_{r, t} \omega_{r, t}\left(\partial_{k}\right)\left[\bar{c}\left(h_{r}\right) \bar{c}\left(h_{t}\right)-c\left(h_{r}\right) c\left(h_{t}\right)\right]
;\nonumber\\
&\xi^{j}=g^{ij}\xi_{i};~~~~\Gamma^{k}=g^{ij}\Gamma_{ij}^{k};~~~~\sigma^{j}=g^{ij}\sigma_{i};~~~\tilde{S}(\partial i)=g^{ij}S(\partial i).
\end{align}

Then we have the following lemmas.
\begin{lem}\label{lem3} The following identities hold:
\begin{align}
 \sigma_{0}(\widetilde{\nabla}^{F}_{X}\widetilde{\nabla}^{F}_{Y})=&X[M(Y)]+X[N(Y)]+X[A(Y)]+M(X)M(Y)+M(X)N(Y)+M(X)A(Y)\nonumber\\
 &+N(X)M(Y)+N(X)N(Y)+N(X)A(Y)+A(X)M(Y)+A(X)N(Y)+A(X)A(Y);\\
\sigma_{1}(\widetilde{\nabla}^{F}_{X}\widetilde{\nabla}^{F}_{Y})
=&\sqrt{-1}\sum_{j,l=1}^nX_j\frac{\partial_{Y_l}}{\partial_{x_j}}\xi_l
+\sqrt{-1}\sum_{j=1}^{n}[M(Y)+N(Y)+A(Y)]X_j\xi_j\nonumber\\
&+\sqrt{-1}\sum_{l=1}^{n}[M(X)+N(X)+A(X)]Y_l\xi_l;\\
\sigma_{2}(\widetilde{\nabla}^{F}_{X}\widetilde{\nabla}^{F}_{Y})=&-\sum_{j,l=1}^nX_jY_l\xi_j\xi_l.
\end{align}
\end{lem}

\begin{lem} \label{lem4}\cite{WW} The following identities hold:
\begin{align}
\sigma_{-1}(D_{F}^{-1})= & \frac{ic(\xi)}{|\xi|^2};\\
\sigma_{-2}(D_{F}^{-1})= & \frac{c(\xi)\sigma_{0}(D_F)c(\xi)}{|\xi|^4}
+\frac{c(\xi)}{|\xi|^6}\sum_{j}c(\mathrm{d}x_j)[\partial_{x_j}(c(\xi))|\xi|^2-c(\xi)\partial_{x_j}(|\xi|^2)];\\
\sigma_{-2}\left(D_{F}^{-2}\right)= & |\xi|^{-2};\\
\sigma_{-3}\left(D_{F}^{-2}\right)= & -\sqrt{-1}|\xi|^{-4} \xi_{k}\left(\Gamma^{k}-2 \sigma^{k} \otimes \operatorname{Id}_{\wedge\left(F^{\perp, \star}\right)}-\operatorname{Id}_{S(F)} \otimes 2 \tilde{\sigma}^{k}\right. \nonumber\\
& \left.-\frac{1}{2} \sum_{i, j=1}^{2 p} \sum_{s=1}^{q}\left\langle\nabla_{\partial_{k}}^{T M} f_{j}, h_{s}\right\rangle c\left(f_{j}\right) c\left(h_{s}\right)-\frac{1}{2} \sum_{s, t=1}^{q} \sum_{i=1}^{2 p}\left\langle\nabla_{\partial_{k}}^{T M} f_{j}, h_{s}\right\rangle c\left(f_{j}\right) c\left(h_{s}\right)\right)\nonumber \\
& -\sqrt{-1}|\xi|^{-6} 2 \xi^{j} \xi_{\alpha} \xi_{\beta} \partial_{j} g^{\alpha \beta} .
\end{align}
where,
    \begin{align}
\sigma_{0}\left(D_{F}\right)= & -\frac{1}{4} \sum_{i, k, l} \omega_{k, l}\left(f_{i}\right) c\left(f_{i}\right) c\left(f_{k}\right) c\left(f_{l}\right) \otimes \operatorname{Id}_{\wedge\left(F^{\perp, \star}\right)} \nonumber\\
& -\frac{1}{4} \sum_{s, k, l} \omega_{k, l}\left(f_{i}\right) c\left(f_{k}\right) c\left(f_{l}\right) c\left(h_{s}\right) \otimes \operatorname{Id}_{\wedge\left(F^{\perp, \star}\right)} \nonumber\\
& +\operatorname{Id}_{S(F)} \otimes \frac{1}{4} \sum_{i, r, t} \omega_{r, t}\left(f_{i}\right) c\left(f_{i}\right)\left[\bar{c}\left(h_{r}\right) \bar{c}\left(h_{t}\right)-c\left(h_{r}\right) c\left(h_{t}\right)\right] \nonumber\\
& +\operatorname{Id}_{S(F)} \otimes \frac{1}{4} \sum_{s, r, t} \omega_{r, t}\left(h_{s}\right) c\left(h_{s}\right)\left[\bar{c}\left(h_{r}\right) \bar{c}\left(h_{t}\right)-c\left(h_{r}\right) c\left(h_{t}\right)\right] \nonumber\\
& +\frac{1}{2} \sum_{i, j=1}^{p} \sum_{s=1}^{q}\left\langle\nabla_{f_{i}}^{T M} f_{j}, h_{s}\right\rangle c\left(f_{i}\right) c\left(f_{j}\right) c\left(h_{s}\right)+\frac{1}{2} \sum_{s, t=1}^{q} \sum_{i=1}^{p}\left\langle\nabla_{h_{s}}^{T M} h_{t}, f_{i}\right\rangle c\left(h_{s}\right) c\left(h_{t}\right) c\left(f_{i}\right) .\nonumber
   \end{align}
\end{lem}

\begin{lem} The following identities hold:
\begin{align}
\sigma_{0}(\widetilde{\nabla}_{X}^{F}\widetilde{\nabla}_{Y}^{F}D^{-2}_{F})=&
-\sum_{j,l=1}^nX_jY_l\xi_j\xi_l|\xi|^{-2};\\
\sigma_{-1}(\widetilde{\nabla}^{F}_{X}\widetilde{\nabla}^{F}_{Y}D^{-2}_{F})=&
\sigma_{2}(\widetilde{\nabla}^{F}_{X}\widetilde{\nabla}^{F}_{Y})\sigma_{-3}(D^{-2}_{F})
+\sigma_{1}(\widetilde{\nabla}^{F}_{X}\widetilde{\nabla}^{F}_{Y})\sigma_{-2}((D^{-2}_{F})\nonumber\\
&+\sum_{j=1}^{n}\partial_{\xi_{j}}\big[\sigma_{2}(\widetilde{\nabla}^{F}_{X}\widetilde{\nabla}^{F}_{Y})\big]
D_{x_{j}}\big[\sigma_{-2}((D^{-2}_{F})\big].
\end{align}
\end{lem}

Now we  need to compute $\int_{\partial M} \Phi$. When $n=p+q=4$, then ${\rm tr}_{S(F)\otimes(F^\bot,*)}[{\rm \texttt{id}}]=8$, the sum is taken over $
r+l-k-j-|\alpha|=-3,~~r\leq 0,~~l\leq-2,$ then we have the following five cases:

 {\bf case a)~I)}~$r=0,~l=-2,~k=j=0,~|\alpha|=1$.

 By (3.4), we get
\begin{equation}
\label{b24}
\Phi_1=-\int_{|\xi'|=1}\int^{+\infty}_{-\infty}\sum_{|\alpha|=1}
\mathrm{Tr}[\partial^\alpha_{\xi'}\pi^+_{\xi_n}\sigma_{0}(\widetilde{\nabla}^{F}_{X}\widetilde{\nabla}^{F}_{Y}D^{-2}_{F})\times
 \partial^\alpha_{x'}\partial_{\xi_n}\sigma_{-2}(D^{-2}_{F})](x_0)d\xi_n\sigma(\xi')dx'.
\end{equation}
By Lemma 2.2 in \cite{Wa3}, for $i<n$, then
\begin{equation}
\label{b25}
\partial_{x_i}\sigma_{-2}(D^{-2}_{F})(x_0)=
\partial_{x_i}(|\xi|^{-2})(x_0)=
-\frac{\partial_{x_i}(|\xi|^{2})(x_0)}{|\xi|^4}=0,
\end{equation}
 so $\Phi_1=0$.

 {\bf case a)~II)}~$r=0,~l=-2,~k=|\alpha|=0,~j=1$.

 By (3.4), we get
\begin{equation}
\label{b26}
\Phi_2=-\frac{1}{2}\int_{|\xi'|=1}\int^{+\infty}_{-\infty}
\mathrm{Tr}[\partial_{x_n}\pi^+_{\xi_n}\sigma_{0}(\widetilde{\nabla}^{F}_{X}\widetilde{\nabla}^{F}_{Y}D_{F}^{-2})\times
\partial_{\xi_n}^2\sigma_{-2}(D^{-2}_{F})](x_0)d\xi_n\sigma(\xi')dx'.
\end{equation}
By Lemma 3.3, we have
\begin{eqnarray}\label{b237}
\partial_{\xi_n}^2\sigma_{-2}(D^{-2}_{F})(x_0)=\partial_{\xi_n}^2(|\xi|^{-2})(x_0)=\frac{6\xi_n^2-2}{(1+\xi_n^2)^3}.
\end{eqnarray}
It follows that
\begin{eqnarray}\label{b27}
\partial_{x_n}\sigma_{0}(\widetilde{\nabla}^{F}_{X}\widetilde{\nabla}^{F}_{Y}D^{-2}_{F})(x_0)
=\partial_{x_n}(-\sum_{j,l=1}^nX_jY_l\xi_j\xi_l|\xi|^{-2})=\frac{\sum_{j,l=1}^nX_jY_l\xi_j\xi_lh'(0)|\xi'|^2}{(1+\xi_n^2)^2}.
\end{eqnarray}
By integrating formula, we obtain
\begin{align}\label{b28}
\pi^+_{\xi_n}\partial_{x_n}\sigma_{0}(\widetilde{\nabla}^{F}_{X}\widetilde{\nabla}^{F}_{Y}D^{-2}_{F})(x_0)
&=\partial_{x_n}\pi^+_{\xi_n}\sigma_{0}(\widetilde{\nabla}^{F}_{X}\widetilde{\nabla}^{F}_{Y}(D^{-2}_{F})\nonumber\\
&=-\frac{i\xi_n+2}{4(\xi_n-i)^2}\sum_{j,l=1}^{n-1}X_jY_l\xi_j\xi_lh'(0)-\frac{i\xi_n}{4(\xi_n-i)^2}X_nY_nh'(0)\nonumber\\
&-\frac{i}{4(\xi_n-i)^2}\sum_{j=1}^{n-1}X_jY_n\xi_j-\frac{i}{4(\xi_n-i)^2}\Sigma_{l=1}^{n-1}X_nY_l\xi_l.
\end{align}
We note that $i<n,~\int_{|\xi'|=1}\xi_{i_{1}}\xi_{i_{2}}\cdots\xi_{i_{2d+1}}\sigma(\xi')=0$,
so we omit some items that have no contribution for computing {\bf case a)~II)}.
From (3.32) and (3.34), we obtain
\begin{align}\label{33}
&\mathrm{Tr} [\partial_{x_n}\pi^+_{\xi_n}\sigma_{0}(\widetilde{\nabla}^{F}_{X}\widetilde{\nabla}^{F}_{Y}D_{F}^{-2})\times
\partial_{\xi_n}^2\sigma_{-2}(D_{F}^{-2})](x_0)\nonumber\\
&=4[\frac{i-3i\xi_n^2}{(\xi_n-i)^4(\xi_n+i)^3}+\frac{1-3\xi_n^2}{(\xi_n-i)^5(\xi_n+i)^3}]\sum_{j,l=1}^{n-1}X_jY_l\xi_j\xi_lh'(0)\nonumber\\
&+4[\frac{i-3i\xi_n^2}{(\xi_n-i)^4(\xi_n+i)^3}-\frac{1-3\xi_n^2}{(\xi_n-i)^5(\xi_n+i)^3}]X_nY_nh'(0)\nonumber\\
&+4\frac{(1-3\xi_n^2)i}{(\xi_n-i)^5(\xi_n+i)^3}\sum_{j=1}^{n-1}X_jY_n\xi_j
+4\frac{(1-3\xi_n^2)i}{(\xi_n-i)^5(\xi_n+i)^3}\sum_{l=1}^{n-1}X_nY_l\xi_l.
\end{align}
Therefore, we get
\begin{align}\label{35}
\Phi_2
&=\left(\frac{5}{24}\sum_{j=1}^{n-1}X_jY_j-\frac{1}{8}X_nY_n\right)h'(0)\pi\Omega_3dx',
\end{align}
where ${\rm \Omega_{3}}$ is the canonical volume of $S^{2}.$

  {\bf case a)~III)}~$r=0,~l=-2,~j=|\alpha|=0,~k=1$.

By (3.4), we get
\begin{align}\label{36}
\Phi_3&=-\frac{1}{2}\int_{|\xi'|=1}\int^{+\infty}_{-\infty}
\mathrm{Tr} [\partial_{\xi_n}\pi^+_{\xi_n}\sigma_{0}(\widetilde{\nabla}^{F}_{X}\widetilde{\nabla}^{F}_{Y}D_{F}^{-2})\times
\partial_{\xi_n}\partial_{x_n}\sigma_{-2}(D_{F}^{-2})](x_0)d\xi_n\sigma(\xi')dx'\nonumber\\
&=\frac{1}{2}\int_{|\xi'|=1}\int^{+\infty}_{-\infty}
\mathrm{Tr}[\partial_{\xi_n}^2\pi^+_{\xi_n}\sigma_{0}(\widetilde{\nabla}^{F}_{X}\widetilde{\nabla}^{F}_{Y}D_{F}^{-2})\times
\partial_{x_n}\sigma_{-2}(D_{F}^{-2})](x_0)d\xi_n\sigma(\xi')dx'.
\end{align}
 By Lemma 4.4, we have
\begin{eqnarray}\label{37}
\partial_{x_n}\sigma_{-2}(D_{F}^{-2})(x_0)|_{|\xi'|=1}
=-\frac{h'(0)}{(1+\xi_n^2)^2}.
\end{eqnarray}
An easy calculation gives
\begin{align}\label{38}
\pi^+_{\xi_n}\sigma_{0}(\widetilde{\nabla}^{F}_{X}\widetilde{\nabla}^{F}_{Y}(D^{-2}_{F})(x_0)|_{|\xi'|=1}&
=\frac{i}{2(\xi_n-i)}\sum_{j,l=1}^{n-1}X_jY_l\xi_j\xi_l-\frac{i}{2(\xi_n-i)}X_nY_n\nonumber\\
&-\frac{1}{2(\xi_n-i)}\sum_{j=1}^{n-1}X_jY_n\xi_j-\frac{1}{2(\xi_n-i)}\sum_{l=1}^{n-1}X_nY_l\xi_l.
\end{align}
Also, straightforward computations yield
\begin{align}\label{mmmmm}
\partial_{\xi_n}^2\pi^+_{\xi_n}\sigma_{0}(\widetilde{\nabla}^{F}_{X}\widetilde{\nabla}^{F}_{Y}(D^{-2}_{F})(x_0)|_{|\xi'|=1}
&=\frac{i}{(\xi_n-i)^3}\sum_{j,l=1}^{n-1}X_jY_l\xi_j\xi_l-\frac{1}{(\xi_n-i)^3}X_nY_n\nonumber\\
&-\frac{1}{(\xi_n-i)^3}\sum_{j=1}^{n-1}X_jY_n\xi_j-\frac{1}{(\xi_n-i)^3}\sum_{l=1}^{n-1}X_nY_l\xi_l.
\end{align}

Therefore, we get
\begin{align}\label{41}
\Phi_3&=\frac{1}{2}\int_{|\xi'|=1}\int^{+\infty}_{-\infty}
\mathrm{Tr}[\partial^2_{\xi_n}\pi^+_{\xi_n}\sigma_{0}(\widetilde{\nabla}^{F}_{X}\widetilde{\nabla}^{F}_{Y}(D^{-2}_{F})\times
\partial_{x_n}\sigma_{-2}(D^{-2}_{F})](x_0)d\xi_n\sigma(\xi')dx'\nonumber\\
&=\left(-\frac{5}{24}\sum_{j=1}^{n-1}X_jY_j+\frac{5}{8}X_nY_n\right)h'(0)\pi\Omega_3dx'.
\end{align}

 {\bf case b)}~$r=0,~l=-3,~k=j=|\alpha|=0$.

 By (3.4), we get
\begin{align}\label{42}
\Phi_4&=-i\int_{|\xi'|=1}\int^{+\infty}_{-\infty}\mathrm{Tr} [\pi^+_{\xi_n}
\sigma_{0}(\widetilde{\nabla}^{F}_{X}\widetilde{\nabla}^{F}_{Y}D^{-2}_{F})\times
\partial_{\xi_n}\sigma_{-3}(D^{-2}_{F})](x_0)d\xi_n\sigma(\xi')dx'\nonumber\\
&=i\int_{|\xi'|=1}\int^{+\infty}_{-\infty}\mathrm{Tr} [\partial_{\xi_n}\pi^+_{\xi_n}
\sigma_{0}(\widetilde{\nabla}^{F}_{X}\widetilde{\nabla}^{F}_{Y}D^{-2}_{F})\times
\sigma_{-3}(D^{-2}_{F})](x_0)d\xi_n\sigma(\xi')dx'.
\end{align}
 By Lemma 3.4, we have
\begin{align}\label{43}
&\sigma_{-3}\left(D_{F}^{-2}\right)(x_0)|_{|\xi'|=1}= \frac{-2 i h^{\prime}(0) \xi_{n}}{\left(1+\xi_{n}^{2}\right)^{3}}-\frac{i}{\left(1+\xi_{n}^{2}\right)^{2}}\left(\frac{3}{2}h^{\prime}(0) \xi_{n}+\frac{1}{2} \sum_{k, l} \omega_{k, l}\left(\partial^{k}\right)(x_0) c\left(f_{k}\right) c\left(f_{l}\right) \otimes \operatorname{Id}_{\wedge\left(F^{\perp, \star}\right)}\right. \nonumber\\
& \left.-\operatorname{Id}_{S(F)} \otimes \frac{1}{2} \sum_{r, t} \omega_{r, t}\left(\partial^{k}\right)(x_0)\left[\hat{c}\left(h_{r}\right) \hat{c}\left(h_{t}\right)-c\left(h_{r}\right) c\left(h_{t}\right)\right]-\sum_{i, j=1}^{p} \sum_{s=1}^{q}\left\langle\nabla_{\partial^{k}}^{T M} f_{j}, h_{s}\right\rangle c\left(f_{j}\right) c\left(h_{s}\right)\right) .
\end{align}

\begin{align}\label{45}
\partial_{\xi_n}\pi^+_{\xi_n}\sigma_{0}(\widetilde{\nabla}^{F}_{X}\widetilde{\nabla}^{F}_{Y}D^{-2}_{F})(x_0)|_{|\xi'|=1}
&=-\frac{i}{2(\xi_n-i)^2}\sum_{j,l=1}^{n-1}X_jY_l\xi_j\xi_l+\frac{i}{2(\xi_n-i)^2}X_nY_n\nonumber\\
&+\frac{1}{2(\xi_n-i)^2}\sum_{j=1}^{n-1}X_jY_n\xi_j+\frac{1}{2(\xi_n-i)^2}\sum_{l=1}^{n-1}X_nY_l\xi_l.
\end{align}
We note that $i<n,~\int_{|\xi'|=1}\xi_{i_{1}}\xi_{i_{2}}\cdots\xi_{i_{2d+1}}\sigma(\xi')=0$,
so we omit some items that have no contribution for computing {\bf case b)}.
By the trace identity  $\operatorname{tr}(A B)=\operatorname{tr}(B A), \operatorname{tr}(A \otimes B)=\operatorname{tr}(A) \cdot \operatorname{tr}(B) $ and the relation of the Clifford action, we have

\begin{align}
	&\mathrm{Tr}\left[c\left(f_{k}\right) c\left(f_{l}\right)\right]=-\delta_{k}^{l} 2^{\frac{p}{2}}, \\
    &\mathrm{Tr}\left[\hat{c}\left(h_{r}\right) \hat{c}\left(h_{t}\right)-c\left(h_{r}\right) c\left(h_{t}\right)\right]=2 \delta_{r}^{t} 2^{q} .
\end{align}

Then, we have
\begin{align}\label{39}
&\mathrm{Tr}[\partial_{\xi_n}\pi^+_{\xi_n}\sigma_{0}(\widetilde{\nabla}^{F}_{X}\widetilde{\nabla}^{F}_{Y}D^{-2}_{F})\times
\sigma_{-3}D^{-2}_{F})](x_0)\nonumber\\
&\sim[-\frac{8\xi_n}{(\xi_n-i)^5(\xi_n+i)^3}-\frac{6\xi_n}{(\xi_n-i)^4(\xi_n+i)^2}]h'(0)\sum_{j,l=1}^{n-1}X_jY_l\xi_j\xi_l\nonumber\\
&~~+[\frac{8\xi_n}{(\xi_n-i)^5(\xi_n+i)^3}+\frac{6\xi_n}{(\xi_n-i)^4(\xi_n+i)^2}]h'(0)X_nY_n.
\end{align}
Therefore, we get
\begin{align}\label{41}
&\Phi_4=i\int_{|\xi'|=1}\int^{+\infty}_{-\infty}
\mathrm{Tr}[\partial_{\xi_n}\pi^+_{\xi_n}\sigma_{0}(\widetilde{\nabla}_{X}^{F}\widetilde{\nabla}_{Y}^{F}D_{F}^{-2})\times
\sigma_{-3}D_{F}^{-2})](x_0)d\xi_n\sigma(\xi')dx'\nonumber\\
&=\left(\frac{11}{24}\sum_{j=1}^{n-1}X_jY_j-\frac{11}{8}X_nY_n\right)h'(0)\pi\Omega_3dx'.
\end{align}

 {\bf  case c)}~$r=-1,~\ell=-2,~k=j=|\alpha|=0$.

By (3.4), we get
\begin{align}\label{61}
\Phi_5=-i\int_{|\xi'|=1}\int^{+\infty}_{-\infty}\mathrm{Tr} [\pi^+_{\xi_n}
\sigma_{-1}(\widetilde{\nabla}_{X}^{F}\widetilde{\nabla}_{Y}^{F}D_{F}^{-2})\times
\partial_{\xi_n}\sigma_{-2}(D_{F}^{-2})](x_0)d\xi_n\sigma(\xi')dx'.
\end{align}
By Lemma 4.4, we have
\begin{align}\label{62}
\partial_{\xi_n}\sigma_{-2}(D_{F}^{-2})(x_0)|_{|\xi'|=1}=-\frac{2\xi_n}{(\xi_n^2+1)^2}.
\end{align}
Since
\begin{align}
\sigma_{-1}(\widetilde{\nabla}_{X}^{F}\widetilde{\nabla}_{Y}^{F}D_{F}^{-2}))(x_0)|_{|\xi'|=1}
=&\sigma_{2}(\widetilde{\nabla}_{X}^{F}\widetilde{\nabla}_{Y}^{F})\sigma_{-3}(D_{F}^{-2})
+\sigma_{1}(\widetilde{\nabla}_{X}^{F}\widetilde{\nabla}_{Y}^{F})\sigma_{-2}(D_{F}^{-2})\nonumber\\
&+\sum_{j=1}^{n}\partial_{\xi_{j}}\big[\sigma_{2}(\widetilde{\nabla}_{X}^{F}\widetilde{\nabla}_{Y}^{F})\big]
D_{x_{j}}\big[\sigma_{-2}(D_{F}^{-2})\big].
\end{align}
 Explicit representation the first item of (3.41),
\begin{align}
&\sigma_{2}(\widetilde{\nabla}_{X}^{F}\widetilde{\nabla}_{Y}^{F})\sigma_{-3}(D_{F}^{-2})(x_0)|_{|\xi'|=1}\nonumber\\
=&-\sum_{j,l=1}^{n}X_jY_l\xi_j\xi_l\times \bigg \{\frac{-2 i h^{\prime}(0) \xi_{n}}{\left(1+\xi_{n}^{2}\right)^{3}}-\frac{i}{\left(1+\xi_{n}^{2}\right)^{2}}\left(\frac{3}{2}h^{\prime}(0) \xi_{n}+\frac{1}{2} \sum_{k, l} \omega_{k, l}\left(\partial^{k}\right)(x_0) c\left(f_{k}\right) c\left(f_{l}\right) \otimes \operatorname{Id}_{\wedge\left(F^{\perp, \star}\right)}\right. \nonumber\\
& \left.-\operatorname{Id}_{S(F)} \otimes \frac{1}{2} \sum_{r, t} \omega_{r, t}\left(\partial^{k}\right)(x_0)\left[\hat{c}\left(h_{r}\right) \hat{c}\left(h_{t}\right)-c\left(h_{r}\right) c\left(h_{t}\right)\right]-\sum_{i, j=1}^{p} \sum_{s=1}^{q}\left\langle\nabla_{\partial^{k}}^{T M} f_{j}, h_{s}\right\rangle c\left(f_{j}\right) c\left(h_{s}\right)\right) \bigg \}.
\end{align}
Explicit representation the second item of (3.41),
\begin{align}
&\sigma_{1}(\widetilde{\nabla}_{X}^{F}\widetilde{\nabla}_{Y}^{F})\sigma_{-2}(D_{F}^{-2})(x_0)|_{|\xi'|=1}\nonumber\\
=&\Big(\sqrt{-1}\sum_{j,l=1}^nX_j\frac{\partial_{Y_l}}{\partial_{x_j}}\xi_l
+\sqrt{-1}\sum_j[M(Y)+N(Y)+A(Y)]X_j\xi_j\nonumber\\
+&\sqrt{-1}\sum_l[M(X)+N(X)+A(X)]Y_l\xi_l\Big)\times|\xi|^{-2}.
\end{align}
Explicit representation the third item of (3.41),
\begin{align}
&\sum_{j=1}^{n}\partial_{\xi_{j}}\big[\sigma_{2}(\widetilde{\nabla}_{X}^{F}\widetilde{\nabla}_{Y}^{F})\big]
D_{x_{j}}\big[\sigma_{-2}(D_{F}^{-2})\big](x_0)|_{|\xi'|=1}\nonumber\\
=&\sum_{j=1}^{n}\partial_{\xi_{j}}\big[-\sum_{j,l=1}^nX_jY_l\xi_j\xi_l\big]
(-\sqrt{-1})\partial_{x_{j}}\big(|\xi|^{-2}\big])\nonumber\\
=&\sum_{j=1}^{n}\sum_{l=1}^{n}\sqrt{-1}(x_{j}Y_l+x_{l}Y_j)\xi_{l}\partial_{x_{j}}(|\xi|^{-2}).
\end{align}
We note that $i<n,~\int_{|\xi'|=1}\xi_{i_{1}}\xi_{i_{2}}\cdots\xi_{i_{2d+1}}\sigma(\xi')=0$,
so we omit some items that have no contribution for computing {\bf case c)}.
Also, straightforward computations yield
\begin{align}\label{71}
&\mathrm{Tr} [\pi^+_{\xi_n}
\sigma_{-1}(\widetilde{\nabla}_{X}^{F}\widetilde{\nabla}_{Y}^{F}D_{F}^{-2})\times
\partial_{\xi_n}\sigma_{-2}D_{F}^{-2})](x_0)|_{|\xi'|=1}\nonumber\\
=&\sum_{j,l=1}^{n-1}X_jY_l\xi_j\xi_l\times\frac{18i\xi_n-10\xi_n^2}{(\xi_n-i)^5(\xi_n+i)^2}h'(0)+X_nY_n\times \big[\frac{-20i\xi_n^3-36\xi_n^2+12i\xi_n}{(\xi_n-i)^5(\xi_n+i)^2}+ \frac{8\xi_n}{(\xi_n-i)^4(\xi_n+i)^2}\big]h'(0) ,
\end{align}
 
Hence in this case,
\begin{align}\label{74}
\Phi_5=&-i\int_{|\xi'|=1}\int^{+\infty}_{-\infty}\mathrm{Tr} [\pi^+_{\xi_n}
\sigma_{-1}(\widetilde{\nabla}_{X}^{F}\widetilde{\nabla}_{Y}^{F}D_{F}^{-2})\times
\partial_{\xi_n}\sigma_{-2}D_{F}^{-2})](x_0)d\xi_n\sigma(\xi')dx'\nonumber\\ =&-\left(\frac{2}{3}\sum_{j=1}^{n-1}X_jY_j+\frac{5}{8}X_nY_n\right)h'(0)\pi\Omega_3dx'.
\end{align}
Now $\Phi$ is the sum of the cases (a), (b) and (c). Let $X=X^T+X_n\partial_n,~Y=Y^T+Y_n\partial_n,$ then we have $\sum_{j=1}^{n-1}X_jY_j(x_0)=g(X^T,Y^T)(x_0).$Therefore, we get
\begin{align}\label{795}
\Phi=&\sum_{i=1}^5\Phi_i=-\left[\frac{5}{24}g(X^T,Y^T)+\frac{3}{2}X_nY_n\right]h'(0)\pi\Omega_3dx'.
\end{align}
Then we obtain following theorem
\begin{thm}\label{thmb1}
 Let $M$ be a 4-dimensional compact foliated manifold with boundary and spin leave , $\widetilde{\nabla}$ be an orthogonal
connection with torsion. Then we get the volumes  associated to $\widetilde{\nabla}_{X}^{F}\widetilde{\nabla}_{Y}^{F}D_{F}^{-2}$
and $D_{F}^{-2}$ 
\begin{align}
\label{b2773}
&\widetilde{{\rm Wres}}[\pi^+(\widetilde{\nabla}_{X}^{F}\widetilde{\nabla}_{Y}^{F}D_{F}^{-2})\circ\pi^+(D_{F}^2)]\nonumber\\
=&\frac{8\pi^2}{3}\int_{M}G(X,Y) dvol_{g}+\int_{M}sg(X, Y)dvol_{M}-\int_{\partial M}\left[\frac{5}{24}g(X^T,Y^T)+\frac{3}{2}X_nY_n\right]h'(0)\pi\Omega_3dx'.
\end{align}
where  $s$ is the scalar curvature.
\end{thm}

By \cite{Wa3}, we have the extrinsic curvature $K=-\frac{3}{2}h'(0)$ , so when $X_n=0$ or $Y_n=0$ on boundary, we have 
\begin{cor}
	\begin{align}
		\label{b2774}
		&\widetilde{{\rm Wres}}[\pi^+(\widetilde{\nabla}_{X}^{F}\widetilde{\nabla}_{Y}^{F}D_{F}^{-2})\circ\pi^+(D_{F}^2)]\nonumber\\
		=&\frac{8\pi^2}{3}\int_{M}G(X,Y) dvol_{g}+\int_{M}sg(X, Y)dvol_{M}+\frac{5}{36}\int_{\partial M}g(X,Y)_{\partial M}Kdvol_{\partial M}.
	\end{align}
\end{cor}

So we have
\begin{defn} The spectral Einstein functional for manifolds with boundary
	\begin{equation}\label{}
	E_H:=\int_{M}G(X,Y) dvol_{g}+c_0\int_{\partial M}g(X,Y)_{\partial M}Kdvol_{\partial M},
	\end{equation}
	where $c_0$ is a constant.
\end{defn}

 \section{The residue for  sub-Dirac operators  $\widetilde{\nabla}_{X}^{F}\widetilde{\nabla}_{Y}^{F}D_{F}^{-1}$ and $D_{F}^{-3}$ }

In this section, we compute the 4-dimension volume for sub-Dirac operators
 $\widetilde{\nabla}_{X}^{F}\widetilde{\nabla}_{Y}^{F}D_{F}^{-1}$ and $D_{F}^{-3}$ .
 Since $[\sigma_{-4}(\widetilde{\nabla}_{X}^{F}\widetilde{\nabla}_{Y}^{F}D_{F}^{-1}
  \circ D_{F}^{-3})]|_{M}$ has the same expression as $[\sigma_{-4}(\widetilde{\nabla}_{X}^{F}\widetilde{\nabla}_{Y}^{F}D_{F}^{-1}
  \circ D_{F}^{-3})]|_{M}$ in the case of manifolds without boundary,
so locally we can use Theorem 2.4 to compute the first term.
\begin{cor}
 Let M be a four dimensional compact manifold without boundary and $\widetilde{\nabla}$ be an orthogonal
connection with torsion. Then we get the volumes  associated to $\widetilde{\nabla}_{X}^{F}\widetilde{\nabla}_{Y}^{F}D_{F}^{-1}$
and $D_{F}^{-3}$ on compact manifolds without boundary
 \begin{align}
&Wres[\sigma_{-4}(\widetilde{\nabla}_{X}^{F}\widetilde{\nabla}_{Y}^{F}D_{F}^{-1}
  \circ D_{F}^{-3})]\nonumber\\
=&\frac{8\pi^2}{3}\int_{M}G(X,Y) dvol_{g}+\int_{M}sg(X, Y)dvol_{M},
\end{align}
where $s$ is the scalar curvature.
\end{cor}

Similar definition3.2 ,we have
\begin{align}\label{c4}
	\widetilde\Phi=&\int_{|\xi'|=1}\int_{-\infty}^{+\infty}\sum_{j,k=0}^{\infty}\sum \frac{(-i)^{|\alpha|+j+k+l}}{\alpha!(j+k+1)!}
	\mathrm{tr}_{S(F)\otimes\wedge(F^\bot,*)}[\partial_{x_{n}}^{j}\partial_{\xi'}^{\alpha}\partial_{\xi_{n}}^{k}\sigma_{r}^{+}
	(\nabla_{X}^{F}\nabla_{Y}^{F}(D_{F})^{-p_{1}})(x',0,\xi',\xi_{n})\nonumber\\
	&\times\partial_{x_{n}}^{\alpha}\partial_{\xi_{n}}^{j+1}\partial_{x_{n}}^{k}\sigma_{l}((D_{F}^3)^{-p_{2}})(x',0,\xi',\xi_{n})]
	\texttt{d}\xi_{n}\sigma(\xi')\texttt{d}x' ,
\end{align}

From lemma \ref{lem3} and lemma \ref{lem4}, we have
\begin{lem} The following identities hold:
\begin{align}
\sigma_{1}(\widetilde{\nabla}_{X}^{F}\widetilde{\nabla}_{Y}^{F} D_{F}^{-1})=&
-\sqrt{-1}\sum_{j,l=1}^nX_jY_l\xi_j\xi_lc(\xi)|\xi|^{-2};\\
\sigma_{0}(\widetilde{\nabla}_{X}^{F}\widetilde{\nabla}_{Y}^{F}D_{F}^{-1})=&
\sigma_{2}(\widetilde{\nabla}_{X}^{F}\widetilde{\nabla}_{Y}^{F})\sigma_{-2}(D_{F}^{-1})
+\sigma_{1}(\widetilde{\nabla}_{X}^{F}\widetilde{\nabla}_{Y}^{F})\sigma_{-1}(D_{F}^{-1})\nonumber\\
&+\sum_{j=1}^{n}\partial _{\xi_{j}}\big[\sigma_{2}(\widetilde{\nabla}_{X}^{F}\widetilde{\nabla}_{Y}^{F})\big]
D_{x_{j}}\big[\sigma_{-1}(D_{F}^{-1})\big].
\end{align}
\end{lem}

Write
 \begin{eqnarray}
D_x^{\alpha}&=(-i)^{|\alpha|}\partial_x^{\alpha};
~\sigma(D_{F}^3)=p_3+p_2+p_1+p_0;
~\sigma(D_{F}^{-3})=\sum^{\infty}_{j=3}q_{-j}.
\end{eqnarray}
By the composition formula of pseudodifferential operators, we have
\begin{align}
1=\sigma(D_{F}^3\circ D_{F}^{-3})
&=\sum_{\alpha}\frac{1}{\alpha!}\partial^{\alpha}_{\xi}[\sigma(D_{F}^3)
D_x^{\alpha}[\sigma(D_{F}^{-3})]\nonumber\\
&=(p_3+p_2+p_1+p_0)(q_{-3}+q_{-4}+q_{-5}+\cdots)\nonumber\\
&~~~+\sum_j(\partial_{\xi_j}p_3+\partial_{\xi_j}p_2++\partial_{\xi_j}p_1+\partial_{\xi_j}p_0)
(D_{x_j}q_{-3}+D_{x_j}q_{-4}+D_{x_j}q_{-5}+\cdots)\nonumber\\
&=p_3q_{-3}+(p_3q_{-4}+p_2q_{-3}+\sum_j\partial_{\xi_j}p_3D_{x_j}q_{-3})+\cdots,
\end{align}
so
\begin{equation}
q_{-3}=p_3^{-1};~q_{-4}=-p_3^{-1}[p_2p_3^{-1}+\sum_j\partial_{\xi_j}p_3D_{x_j}(p_-3^{-1})].
\end{equation}
Hence by Lemma 4.1 in \cite{WW} and (3.9), we have
\begin{lem} The symbol of the sub-Dirac operator:
\begin{align}
\sigma_{-3}(D_{F}^{-3})&=\sqrt{-1}c(\xi)|\xi|^{-4};\\
\sigma_{-4}(D_{F}^{-3})&=\frac{c(\xi)\sigma_2(D_{F}^{3})c(\xi)}{|\xi|^8}+\frac{\sqrt{-1}c(\xi)}{|\xi|^8}\bigg(|\xi|^4c(dx_n)\partial_{x_n}c(\xi')\nonumber\\
&-2h'(0)c(\mathrm{d}x_n)c(\xi)+2\xi_nc(\xi)\partial{x_n}c(\xi')+4\xi_nh'(0)\bigg),
\end{align}
where,
 \begin{align}
\sigma_2(D_{F}^{3})=&c(\mathrm{d}x_l)\partial_l(g^{i,j})\xi_i\xi_j+2c(\xi)\big[2\sigma^i \otimes \operatorname{Id}_{\wedge\left(F^{\perp, \star}\right)}+\operatorname{Id}_{S(F)}\otimes2\tilde{\sigma}_{i}-\Gamma^k\nonumber\\
&+2\tilde{S}(\partial i)+(M(X)+N(X))|\xi|^2\big]\xi_{k}.
\end{align}
\end{lem}
Now we  need to compute $\int_{\partial M} \widetilde{\Phi}$. When $n=4$, then ${\rm tr}_{S(F)\otimes \wedge\left(F^{\perp, \star}\right)}[{\rm \texttt{id}}]=8$, the sum is taken over $
r+l-k-j-|\alpha|=-3,~~r\leq 0,~~l\leq-2,$ then we have the following five cases:

 {\bf case a)~I)}~$r=1,~l=-2,~k=j=0,~|\alpha|=1$.

By (4.2), we get
\begin{equation}
\label{b24}
\widetilde{\Phi}_1=-\int_{|\xi'|=1}\int^{+\infty}_{-\infty}\sum_{|\alpha|=1}
\mathrm{Tr}[\partial^\alpha_{\xi'}\pi^+_{\xi_n}\sigma_{1}(\widetilde{\nabla}_{X}^{F}\widetilde{\nabla}_{Y}^{F}D^{-1}_{F})\times
 \partial^\alpha_{x'}\partial_{\xi_n}\sigma_{-3}(D^{-3}_{F})](x_0)d\xi_n\sigma(\xi')dx'.
\end{equation}
By Lemma 2.2 in \cite{Wa3}, for $i<n$, then
\begin{equation}
\label{b25}
\partial_{x_i}\sigma_{-3}(D^{-3}_{F})(x_0)=
\partial_{x_i}(\sqrt{-1}c(\xi)|\xi|^{-4})(x_0)=
\sqrt{-1}\frac{\partial_{x_i}c(\xi)}{|\xi|^4}(x_0)
+\sqrt{-1}\frac{c(\xi)\partial_{x_i}(|\xi|^4)}{|\xi|^8}(x_0)
=0,
\end{equation}
\noindent so $\widetilde{\Phi}_1=0$.\\

  {\bf case a)~II)}~$r=1,~l=-3,~k=|\alpha|=0,~j=1$.

  By (4.2), we get
\begin{equation}
\label{b26}
\widetilde{\Phi}_2=-\frac{1}{2}\int_{|\xi'|=1}\int^{+\infty}_{-\infty}
\mathrm{Tr} [\partial_{x_n}\pi^+_{\xi_n}\sigma_{1}(\widetilde{\nabla}_{X}^{F}\widetilde{\nabla}_{Y}^{F}D_{F}^{-1})\times
\partial_{\xi_n}^2\sigma_{-3}(D_{F}^{-1})](x_0)d\xi_n\sigma(\xi')dx'.
\end{equation}
By Lemma 4.3, we have
\begin{eqnarray}\label{b237}
\partial_{\xi_n}^2\sigma_{-3}(D_{F}^{-3})(x_0)=\partial_{\xi_n}^2(c(\xi)|\xi|^{-4})(x_0)
=\sqrt{-1}\frac{(20\xi_n^2-4)c(\xi')+12(\xi_n^3-\xi_n)c(\mathrm{d}x_n)}{(1+\xi_n^2)^4},
\end{eqnarray}
and
\begin{align}\label{b27}
\partial_{x_n}\sigma_{1}(\widetilde{\nabla}_{X}^{F}\widetilde{\nabla}_{Y}^{F}D_{F}^{-1})(x_0)
&=\partial_{x_n}(-\sqrt{-1}\sum_{j,l=1}^nX_jY_l\xi_j\xi_lc(\xi)|\xi|^{-2})\nonumber\\
&=-\sqrt{-1}\sum_{j,l=1}^nX_jY_l\xi_j\xi_l \left[ \frac{\partial_{x_n}c(\xi')}{1+\xi_n^2}-\frac{c(\xi)h'(0)|\xi'|^2}{(1+\xi_n^2)^2}\right].
\end{align}

We note that $i<n,~\int_{|\xi'|=1}\xi_{i_{1}}\xi_{i_{2}}\cdots\xi_{i_{2d+1}}\sigma(\xi')=0$,
so we omit some items that have no contribution for computing {\bf case a)~II)}.
Then there is the following formula
\begin{align}\label{33}
&\mathrm{Tr}[\partial_{x_n}\pi^+_{\xi_n}\sigma_{1}(\widetilde{\nabla}_{X}^{F}\widetilde{\nabla}_{Y}^{F}D_{F}^{-1})\times
\partial_{\xi_n}^2\sigma_{-3}(D_{F}^{-3})](x_0)\nonumber\\
&=\sum_{j,l=1}^{n-1}X_jY_l\xi_j\xi_lh'(0)\frac{-8(3i\xi_n^2+2\xi_n-i)}{(\xi_n-i)^5(\xi_n+i)^4}\nonumber\\
&+X_nY_nh'(0)\left[\frac{-16i(5\xi_n^2-1)+48(\xi_n^3-\xi_n)}{(\xi_n-i)^5(\xi_n+i)^4}
 +\frac{8(5\xi_n^2-1)+24i(\xi_n^3-\xi_n)}{(\xi_n-i)^6(\xi_n+i)^4}\right].
\end{align}

Therefore, we get
\begin{align}\label{35}
\widetilde{\Phi}_2&=\frac{1}{2}\int_{|\xi'|=1}\int^{+\infty}_{-\infty}\bigg\{
\sum_{j,l=1}^{n-1}X_jY_l\xi_j\xi_lh'(0)\left[\frac{-8(3i\xi_n^2+2\xi_n-i)}{(\xi_n-i)^5(\xi_n+i)^4}\right]\nonumber\\
&+X_nY_nh'(0)\left[\frac{-16i(5\xi_n^2-1)+48(\xi_n^3-\xi_n)}{(\xi_n-i)^5(\xi_n+i)^4}
+\frac{8(5\xi_n^2-1)+24i(\xi_n^3-\xi_n)}{(\xi_n-i)^6(\xi_n+i)^4}\right]
 \bigg\}d\xi_n\sigma(\xi')dx'\nonumber\\
&=\left(\frac{5}{16}\sum_{j=1}^{n-1}X_jY_j
+\frac{1}{16}X_nY_n\right)h'(0)\pi\Omega_3dx',
\end{align}
where ${\rm \Omega_{3}}$ is the canonical volume of $S^{2}.$

 {\bf case a)~III)}~$r=1,~l=-3,~j=|\alpha|=0,~k=1$.

 By (4.2), we get
\begin{align}\label{36}
\widetilde{\Phi}_3&=-\frac{1}{2}\int_{|\xi'|=1}\int^{+\infty}_{-\infty}
\mathrm{Tr} [\partial_{\xi_n}\pi^+_{\xi_n}\sigma_{1}(\widetilde{\nabla}_{X}^{F}\widetilde{\nabla}_{Y}^{F}D_{F}^{-1})\times
\partial_{\xi_n}\partial_{x_n}\sigma_{-3}(D_{F}^{-3})](x_0)d\xi_n\sigma(\xi')dx'\nonumber\\
&=\frac{1}{2}\int_{|\xi'|=1}\int^{+\infty}_{-\infty}
\mathrm{Tr} [\partial_{\xi_n}^2\pi^+_{\xi_n}\sigma_{1}(\widetilde{\nabla}_{X}^{F}\widetilde{\nabla}_{Y}^{F}D_{F}^{-1})\times
\partial_{x_n}\sigma_{-3}(D_{F}^{-3})](x_0)d\xi_n\sigma(\xi')dx'.
\end{align}
\noindent By Lemma 4.3, we have
\begin{eqnarray}\label{37}
\partial_{x_n}\sigma_{-3}(D_{F}^{-3})(x_0)|_{|\xi'|=1}
=\frac{i\partial_{x_n}[c(\xi')]}{(1+\xi_n^2)^4}-\frac{2ih'(0)c(\xi)|\xi'|^2_{g^{\partial M}}}{(1+\xi_n^2)^6}.
\end{eqnarray}
And we have
\begin{align}\label{mmmmm}
\partial_{\xi_n}^2\pi^+_{\xi_n}\sigma_{1}(\widetilde{\nabla}_{X}^{F}\widetilde{\nabla}_{Y}^{F}D_{F}^{-1})
&=-\frac{c(\xi')+ic(\mathrm{d}x_n)}{(\xi_n-i)^3}\sum_{j,l=1}^{n-1}X_jY_l\xi_j\xi_l
+\frac{c(\xi')+ic(\mathrm{d}x_n)}{(\xi_n-i)^3}X_nY_n\nonumber\\
&-\frac{ic(\xi')-c(\mathrm{d}x_n)}{(\xi_n-i)^3}\sum_{j=1}^{n-1}X_jY_n\xi_j-\frac{ic(\xi')-c(\mathrm{d}x_n)}{(\xi_n-i)^3}\sum_{l=1}^{n-1}X_nY_l\xi_l.
\end{align}

We note that $i<n,~\int_{|\xi'|=1}\xi_{i_{1}}\xi_{i_{2}}\cdots\xi_{i_{2d+1}}\sigma(\xi')=0$,
so we omit some items that have no contribution for computing {\bf case a)~III)}, then
\begin{align}\label{39}
&\mathrm{Tr} [\partial_{\xi_n}\pi^+_{\xi_n}\sigma_{1}(\widetilde{\nabla}_{X}^{F}\widetilde{\nabla}_{Y}^{F}D_{F}^{-1})\times
\partial_{\xi_n}\partial_{x_n}\sigma_{-3}(D_{F}^{-3})](x_0)|_{|\xi'|=1}\nonumber\\
&=\big[\frac{4i}{(\xi_n-i)^5(\xi_n+i)^2}+\frac{16}{(\xi_n-i)^5(\xi_n+i)^3}\big]h'(0)\sum_{j,l=1}^{n-1}X_jY_l\xi_j\xi_l\nonumber\\
&~~~~-\big[\frac{4i}{(\xi_n-i)^5(\xi_n+i)^2}+\frac{16}{(\xi_n-i)^5(\xi_n+i)^3}\big]h'(0)X_nY_n.\nonumber\\
\end{align}

Therefore, we get
\begin{align}\label{41}
\widetilde{\Phi}_3&=\frac{1}{2}\int_{|\xi'|=1}\int^{+\infty}_{-\infty}
\bigg(\big[\frac{4i}{(\xi_n-i)^5(\xi_n+i)^2}+\frac{16}{(\xi_n-i)^5(\xi_n+i)^3}\big]h'(0)\sum_{j,l=1}^{n-1}X_jY_l\xi_j\xi_l\nonumber\\
&~~~~-\big[\frac{4i}{(\xi_n-i)^5(\xi_n+i)^2}+\frac{16}{(\xi_n-i)^5(\xi_n+i)^3}\big]h'(0)X_nY_n\bigg)d\xi_n\sigma(\xi')dx'\nonumber\\
&=\left(-\frac{25}{48}\sum_{j=1}^{n-1}X_jY_j+\frac{25}{16}X_nY_n\right)h'(0)\pi\Omega_3dx'.
\end{align}

 {\bf case b)}~$r=0,~l=-3,~k=j=|\alpha|=0$.

By (4.2), we get
\begin{align}\label{42}
\widetilde{\Phi}_4&=-i\int_{|\xi'|=1}\int^{+\infty}_{-\infty}
\mathrm{Tr}[\pi^+_{\xi_n}\sigma_{0}(\widetilde{\nabla}_{X}^{F}\widetilde{\nabla}_{Y}^{F}D_{F}^{-1})\times
\partial_{\xi_n}\sigma_{-3}(D_{F}^{-3})](x_0)d\xi_n\sigma(\xi')dx'.
\end{align}
By Lemma 4.3, we obtain
\begin{align}\label{43}
\partial_{\xi_n}\sigma_{-3}(D_{F}^{-3})(x_0)|_{|\xi'|=1}
=\frac{ic(\mathrm{d}x_n)}{(1+\xi_n^2)^2}-\frac{4\sqrt{-1}\xi_nc(\xi)}{(1+\xi_n^2)^3}.
\end{align}
By Lemma 4.2, we have
\begin{align}\label{b41}
\sigma_{0}(\widetilde{\nabla}_{X}^{F}\widetilde{\nabla}_{Y}^{F}D_{F}^{-1})=&
\sigma_{2}(\widetilde{\nabla}_{X}^{F}\widetilde{\nabla}_{Y}^{F})\sigma_{-2}(D_{F}^{-1})
+\sigma_{1}(\widetilde{\nabla}_{X}^{F}\widetilde{\nabla}_{Y}^{F})\sigma_{-1}(D_{F}^{-1})\nonumber\\
&+\sum_{j=1}^{n}\partial _{\xi_{j}}\big[\sigma_{2}(\widetilde{\nabla}_{X}^{F}\widetilde{\nabla}_{Y}^{F})\big]
D_{x_{j}}\big[\sigma_{-1}(D_{F}^{-1})\big]\nonumber\\
&:=A+B+C.
\end{align}

(1) Explicit representation the first item of \eqref{b41}
\begin{align}
&\sigma_{2}(\widetilde{\nabla}_{X}^{F}\widetilde{\nabla}_{Y}^{F})\sigma_{-2}(D_{F}^{-1})(x_0)|_{|\xi'|=1}
=-\sum_{j,l=1}^{n}X_jY_l\xi_j\xi_l\sigma_{-2}(D_{F}^{-1})(x_0)|_{|\xi'|=1}\nonumber\\
&=-\sum_{j,l=1}^{n-1}X_jY_l\xi_j\xi_l\sigma_{-2}(D_{F}^{-1})(x_0)|_{|\xi'|=1}
  -X_nY_n\xi_n^2\sigma_{-2}(D_{F}^{-1})(x_0)|_{|\xi'|=1}\nonumber\\
&-\sum_{j=1}^{n-1}X_jY_n\xi_j\xi_n\sigma_{-2}(D_{F}^{-1})(x_0)|_{|\xi'|=1}
-\sum_{l=1}^{n-1}X_nY_l\xi_n\xi_l\sigma_{-2}(D_{F}^{-1})(x_0)|_{|\xi'|=1},
\end{align}
we let
\begin{align}
	Q_1=\sum_{j,l=1}^{n-1}X_jY_l\xi_j\xi_l\sigma_{-2}(D_{F}^{-1})(x_0)|_{|\xi'|=1},~~
	Q_2=X_nY_n\xi_n^2\sigma_{-2}(D_{F}^{-1})(x_0)|_{|\xi'|=1}.
\end{align}
By the trace identity  $\operatorname{tr}(A B)=\operatorname{tr}(B A), \operatorname{tr}(A \otimes B)=\operatorname{tr}(A) \cdot \operatorname{tr}(B) $ and the relation of the Clifford action, we have

\begin{align}
	&\mathrm{Tr}\left[\sigma_{0}(D_{F})c(h_q)\right]=4\sum_{i=1}^{p+q}\left\langle \nabla_{e_i}^{TM}e_i, h_q\right\rangle=-4\sum_{i=1}^{n-1}\left\langle e_j, \nabla_{e_j}^{TM}\partial_{x_n}\right\rangle := -4 div_{\partial M}(\partial x_n),\\
	&\mathrm{Tr}\left[\sigma_{0}(D_{F})c(\xi')\right]=4\sum_{i=1}^{p+q}\left\langle \nabla_{e_i}^{TM}e_i, \xi_{i}\right\rangle=0.
\end{align}

Then
\begin{align}
	&\mathrm{Tr}[\pi^+_{\xi_n}Q_1\times\partial_{\xi_n}\sigma_{-3}(D_{F}^{-3})]|_{|\xi'|=1}\nonumber\\	=&-\sum_{j,l=1}^{n-1}X_jY_l\xi_j\xi_lh'(0)\left[\frac{3\xi_n^3-6i\xi_n^2-5\xi_n+2i}{(\xi_n-i)^5(\xi_n+i)^3}-\frac{3\xi_n^4-9i\xi_n^3-17\xi_n^2+15i\xi_n+4}{(\xi_n-i)^6(\xi_n+i)^3}\right
	]\nonumber\\
	&-\sum_{j,l=1}^{n-1}X_jY_l\xi_j\xi_lh'(0)\frac{3\xi_n-i}{2(\xi_n-i)^4(\xi_n+i)^3}\mathrm{Tr}\left[i\sigma_{0}(D_{F})c(\mathrm{d}x_n)+\sigma_{0}(D_{F})c(\xi')\right]\nonumber\\
	=&-2\sum_{j,l=1}^{n-1}X_jY_l\xi_j\xi_lh'(0)\frac{3\xi_n-i}{(\xi_n-i)^5(\xi_n+i)^3}+2\sum_{j,l=1}^{n-1}X_jY_l\xi_j\xi_lh'(0)div_{\partial M}(\partial x_n)\frac{3i\xi_n+1}{(\xi_n-i)^4(\xi_n+i)^3},
\end{align}
and
\begin{align}\label{b42}
	&\mathrm{Tr}[\pi^+_{\xi_n}Q_2\times \partial_{\xi_{n}}(D_{F}^{-3})]|_{|\xi'|=1}\nonumber\\
	=&-X_nY_nh'(0)\left[\frac{3\xi_n^2-5i\xi_n}{(\xi_n-i)^4(\xi_n+i)^3}+\frac{3\xi_n^4+6i\xi_n^3+11\xi_n^2-6i\xi_n}{(\xi_n-i)^6(\xi_n+i)^3}\right]\nonumber\\
	&-X_nY_nh'(0)\frac{6i\xi_n^2+5\xi_n-i}{2(\xi_n-i)^4(\xi_n+i)^3}\mathrm{Tr}\left[i\sigma_{0}(D_{F})c(\mathrm{d}x_n)+\sigma_{0}(D_{F})c(\xi')\right]\nonumber\\
	=&-X_nY_nh'(0)\frac{6\xi_n^4-5i\xi_n^3-2\xi_n^2-i\xi_n}{(\xi_n-i)^6(\xi_n+i)^3}-2X_nY_nh'(0)div_{\partial M}(\partial x_n)\frac{6\xi_n^2-5i\xi_n-1}{(\xi_n-i)^4(\xi_n+i)^3}.
\end{align}
 We have,
\begin{align}\label{39}
&-i\int_{|\xi'|=1}\int^{+\infty}_{-\infty}
\mathrm{Tr} [\pi^+_{\xi_n}(Q_1+Q_2)\times
\partial_{\xi_n}\sigma_{-3}(D_F^{-3})](x_0)d\xi_n\sigma(\xi')dx'\nonumber\\
=&-\left(\frac{5 }{16}\sum_{j=1}^{n-1}X_jY_j+\frac{5}{16}X_nY_n\right)h'(0)\pi\Omega_3dx'+\left(\frac{1}{3}\sum_{j=1}^{n-1}X_jY_j+\frac{1}{2}X_nY_n\right)div_{\partial M}(\partial x_n)h'(0)\pi\Omega_3dx'.
\end{align}

(2) Explicit representation the second item of \eqref{b41}
\begin{align}
&\sigma_{1}(\widetilde{\nabla}_{X}^{F}\widetilde{\nabla}_{Y}^{F})\sigma_{-1}(D_{F}^{-1})(x_0)|_{|\xi'|=1}\nonumber\\
=&\Big[\sqrt{-1}\sum_{j,l=1}^nX_j\frac{\partial_{Y_l}}{\partial_{x_j}}\xi_l
+\sqrt{-1}\sum_j\big(M(Y)+N(Y)+A(Y)\big)X_j\xi_j\nonumber\\
&+\sqrt{-1}\sum_l\big(M(X)+N(X)+A(X)\big)Y_l\xi_l\Big]\frac{\sqrt{-1}c(\xi)}{|\xi|^{2}};
\end{align}
We note that $i<n,~\int_{|\xi'|=1}\xi_{i_{1}}\xi_{i_{2}}\cdots\xi_{i_{2d+1}}\sigma(\xi')=0$,
then
\begin{align}\label{39}
&\mathrm{Tr} \left(\pi^+_{\xi_n}\Big(\sigma_{1}(\widetilde{\nabla}_{X}^{F}\widetilde{\nabla}_{Y}^{F})\sigma_{-1}(D_{F}^{-1})\Big)\times
\partial_{\xi_n}\sigma_{-3}(D^{-3})\right)(x_0)\nonumber\\
&=\frac{4(3\xi_n-i)}{(1+\xi_n^2)^3}\sum_{j,l=1}^nX_j\frac{\partial_{Y_l}}{\partial_{x_j}}\xi_l,
\end{align}
so we have
\begin{align}
	&-i\int_{|\xi'|=1}\int^{+\infty}_{-\infty}
	\mathrm{Tr} [\pi^+_{\xi_n}\sigma_{1}(\widetilde{\nabla}_{X}\widetilde{\nabla}_{Y})\sigma_{-1}(D^{-1})\times
	\partial_{\xi_n}\sigma_{-3}(D^{-3})](x_0)d\xi_n\sigma(\xi')dx'\nonumber\\
	&=\frac{3i}{2}X(Y_n)\Omega_3dx',
\end{align}
	where $X(Y_n)$ is $\sum_{j=1}^nX_j\frac{\partial_{Y_n}}{\partial_{x_j}}.$
	
(3) Explicit representation the third item of \eqref{b41}
\begin{align}
\sum_{j=1}^{n}\sum_{|\alpha|=1}\frac{1}{\alpha!}\partial^{|\alpha|}_{\xi}\big[\sigma_{2}(\widetilde{\nabla}_{X}^{F}\widetilde{\nabla}_{Y}^{F})\big]
D_{x_j}\big[\sigma_{-1}(D_{F}^{-1})\big](x_0)|_{|\xi'|=1}
=\sum_{j=1}^{n}\sum_{l=1}^{n}\sqrt{-1}(x_{j}Y_l+x_{l}Y_j)\xi_{l}\partial_{x_{j}}(\frac{\sqrt{-1}c(\xi)}{|\xi|^{2}}).
\end{align}
We note that $i<n,~\int_{|\xi'|=1}\xi_{i_{1}}\xi_{i_{2}}\cdots\xi_{i_{2d+1}}\sigma(\xi')=0$,
then
\begin{align}\label{39}
	&\mathrm{Tr} \left(\pi^+_{\xi_n}\Big(\sum_{j=1}^{n}\sum_{|\alpha|=1}\frac{1}{\alpha!}\partial^{|\alpha|}_{\xi}\big[\sigma_{2}(\widetilde{\nabla}_{X}^{F}\widetilde{\nabla}_{Y}^{F})\big]
	D_{x_j}\big[\sigma_{-1}(D_{F}^{-1})\big]\Big)\times
	\partial_{\xi_n}\sigma_{-3}(D^{-3})\right)(x_0)|_{|\xi'|=1}\nonumber\\
	&=4X_nY_nh'(0)\frac{3\xi_n^2-i\xi_n}{(\xi_n-i)^4(\xi_n+i)^3},
\end{align}
Substituting (4.37) into (4.23) yields
\begin{align}\label{39}
&-i\int_{|\xi'|=1}\int^{+\infty}_{-\infty}
\mathrm{Tr}\Big[\pi^+_{\xi_n}\Big(\sum_{j=1}^{n}\sum_{\alpha}\frac{1}{\alpha!}\partial^{\alpha}_{\xi}
\big[\sigma_{2}(\widetilde{\nabla}_{X}\widetilde{\nabla}_{Y})\big]
D_x^{\alpha}\big[\sigma_{-1}(D^{-1})\big]\Big)
\times
\partial_{\xi_n}\sigma_{-3}(D^{-3})\Big](x_0)d\xi_n\sigma(\xi')dx'\nonumber\\
&=\frac{1}{2}X_{n}Y_n\pi h'(0)\Omega_3dx'.
\end{align}
Summing up (1), (2) and (3) leads to the desired equality
\begin{align}\label{41}
\widetilde{\Phi}_4
=&\left(-\frac{5 }{16}\sum_{j=1}^{n-1}X_jY_j+\frac{3}{16}X_nY_n\right)h'(0)\pi\Omega_3dx'+\frac{3i}{2}X(Y_n)\Omega_3dx'\nonumber\\
&+\left(\frac{1}{3}\sum_{j=1}^{n-1}X_jY_j+\frac{1}{2}X_nY_n\right)div_{\partial M}(\partial x_n)h'(0)\pi\Omega_3dx'.
\end{align}

 {\bf  case c)}~$r=1,~\ell=-4,~k=j=|\alpha|=0$.

By (4.2), we get
\begin{align}\label{61}
\widetilde{\Phi}_5&=-\int_{|\xi'|=1}\int^{+\infty}_{-\infty}\mathrm{Tr} [\pi^+_{\xi_n}
\sigma_{1}(\widetilde{\nabla}_{X}^{F}\widetilde{\nabla}_{Y}^{F}D_{F}^{-1})\times
\partial_{\xi_n}\sigma_{-4}D_{F}^{-3})](x_0)d\xi_n\sigma(\xi')dx'\nonumber\\
&=\int_{|\xi'|=1}\int^{+\infty}_{-\infty}\mathrm{Tr}
[\partial_{\xi_n}\pi^+_{\xi_n}\sigma_{1}(\widetilde{\nabla}_{X}^{F}\widetilde{\nabla}_{Y}^{F}D_{F}^{-1})\times
\sigma_{-4}D_{F}^{-3})](x_0)d\xi_n\sigma(\xi')dx'.
\end{align}
By Lemma 4.3, we have
\begin{align}\label{62}
\sigma_{-4}(D_{F}^{-3})(x_0)|_{|\xi'|=1}=&\frac{c(\xi)\sigma_2(D_{F}^3)c(\xi)}{(\xi_n^2+1)^4}+\frac{ic(\xi)}{(\xi_n^2+1)^4}
\left[|\xi|^4c(dx_n)\partial_{x_n}c(\xi')-2h'(0)c(dx_n)c(\xi)\right.\nonumber\\
&+\left.2\xi_nc(\xi)\partial_{x_n}c(\xi')+4\xi_nh'(0)\right] ,
\end{align}
where
\begin{align}
	\sigma_{2}(D_{F}^3)=&\sum_{i,j,l}c(d_{x_l})(\partial_lg^{ij})+2c(\xi)\big[2\sigma^i \otimes \operatorname{Id}_{\wedge\left(F^{\perp, \star}\right)}+\operatorname{Id}_{S(F)} \otimes 2 \tilde{\sigma}^{k}-\Gamma^{k}+2\widetilde{S}(\partial_{i})\big]\xi_k\nonumber\\
	&+\big[M(X)+N(X)\big]|\xi|^2,
\end{align}
and
\begin{align}\label{621}
\partial_{\xi_n}\pi^+_{\xi_n}\sigma_{1}(\widetilde{\nabla}_{X}^{F}\widetilde{\nabla}_{Y}^{F}D_{F}^{-1})(x_0)|_{|\xi'|=1}
=&\frac{c(\xi')+ic(\mathrm{d}x_n)}{2(\xi_n-i)^2}\sum_{j,l=1}^{n-1}X_jY_l\xi_j\xi_l-\frac{c(\xi')+ic(\mathrm{d}x_n)}{2(\xi_n-i)^2}X_nY_n \nonumber\\
&+\frac{ic(\xi')-c(\mathrm{d}x_n)}{2(\xi_n-i)^2}\sum_{j=1}^{n-1}X_jY_n\xi_j+\frac{ic(\xi')-c(\mathrm{d}x_n)}{2(\xi_n-i)^2}\sum_{l=1}^{n-1}X_nY_l\xi_l.
\end{align}
We note that $i<n,~\int_{|\xi'|=1}\xi_{i_{1}}\xi_{i_{2}}\cdots\xi_{i_{2d+1}}\sigma(\xi')=0$,
so we omit some items that have no contribution for computing {\bf case c)}.
Also, straightforward computations yield
\begin{align}\label{71}
&{\rm tr}\bigg[\partial_{\xi_n}\pi^+_{\xi_n}\sigma_{-1}(\widetilde{\nabla}_{X}^{F}\widetilde{\nabla}_{Y}^{F}D_{F}^{-1})\times \sigma_{-4}(D_F^{-3})\bigg](x_0)|_{|\xi'|=1}\nonumber\\
=&\sum_{j,l=1}^{n-1}X_jY_l\xi_j\xi_lh'(0)\left[\frac{2i\xi_n+2}{(\xi_n-i)^4(\xi_n+i)^2}+\frac{(24+4i)\xi_n^2+(4-32i)\xi_n-8}{(\xi_n-i)^6(\xi_n+i)^4}-\frac{2\xi_n+2i\xi_n^2}{(\xi_n-i)^5(\xi_n+i)^3}\right]\nonumber\\
&-X_nY_nh'(0)\left[\frac{2i\xi_n+2}{(\xi_n-i)^4(\xi_n+i)^2}+\frac{(24+4i)\xi_n^2+(4-32i)\xi_n-8}{(\xi_n-i)^6(\xi_n+i)^4}-\frac{2\xi_n+2i\xi_n^2}{(\xi_n-i)^5(\xi_n+i)^3}\right]\nonumber\\
&+\sum_{j,l=1}^{n-1}X_jY_l\xi_j\xi_l\xi_k\xi_t g(\mathrm{d}{x^t},h_k)\left[\frac{2(i-1)}{(\xi_n-i)^5(\xi_n+i)^3}+\frac{(i-1)}{4(\xi_n-i)^4(\xi_n+i)^2}\right]h'(0)\nonumber\\
&+X_nY_n\xi_k\xi_t g(\mathrm{d}{x^t},h_k)\left[\frac{2(1-i)}{(\xi_n-i)^5(\xi_n+i)^3}+\frac{(1-i)}{4(\xi_n-i)^4(\xi_n+i)^2}\right]h'(0).
\end{align}
We get
\begin{align}\label{74}
\widetilde{\Phi}_5
=&\left[\frac{129-44i }{320}\sum_{j=1}^{n-1}X_jY_j-\frac{245-26i}{96}X_nY_n\right]h'(0)\pi\Omega_3dx' .\nonumber\\
\end{align}

Let $X=X^T+X_n\partial_n,~Y=Y^T+Y_n\partial_n,$ then we have $\sum_{j=1}^{n-1}X_jY_j(x_0)=g(X^T,Y^T)(x_0).$
 Now $\Phi$ is the sum of the cases (a), (b) and (c). Combining with the five cases, this yields
\begin{align}\label{795}
\widetilde{\Phi}=\sum_{i=1}^5\widetilde{\Phi}_i
=&-\left[\frac{113+132i }{960}g(X^T,Y^T)+\frac{71-26i}{96}X_nY_n\right]h'(0)\pi\Omega_3dx'+\frac{3i}{2}X(Y_n)\Omega_3dx'\nonumber\\
&+\left[\frac{1}{3}g(X^T,Y^T)+\frac{1}{2}X_nY_n\right]div_{\partial M}(\partial x_n)h'(0)\pi\Omega_3dx'.
\end{align}
So, we are reduced to prove the following.
\begin{thm}\label{thmb1}
Let $M$ be a 4-dimensional compact filiated manifold with boundary and spin leave,  $\widetilde{\nabla}$ be an orthogonal
connection with torsion. Then we get the spectral Einstein functional associated to $\widetilde{\nabla}_{X}^{F}\widetilde{\nabla}_{Y}^{F}D_{F}^{-1}$
and $D_{F}^{-3}$ on $M$
\begin{align}
\label{b2773}
&\widetilde{{\rm Wres}}[\pi^+(\widetilde{\nabla}_{X}^{F}\widetilde{\nabla}_{Y}^{F}D_{F}^{-1})
\circ\pi^+(D_{F}^{-3})]\nonumber\\
&=\frac{8\pi^2}{3}\int_{M}G(X,Y) dvol_{g}+\int_{M}sg(X, Y)dvol_{M}\nonumber\\
&-\int_{\partial M}\left[\frac{113+132i }{960}g(X^T,Y^T)+\frac{71-26i}{96}X_nY_n\right]h'(0)\pi\Omega_3dx'+\int_{\partial M}\frac{3i}{2}X(Y_n)\Omega_3dx'\nonumber\\
&+\int_{\partial M}\left[\frac{1}{3}g(X^T,Y^T)+\frac{1}{2}X_nY_n\right]div_{\partial M}(\partial x_n)h'(0)\pi\Omega_3dx'.
\end{align}
where $s$ is the scalar curvature.
\end{thm}

\section*{ Acknowledgements}
This work is supported by the National Natural Science
Foundation of China 11771070, 12061078 and the Younth Scientific Research Project of Yili Normal University 2023YSQN002.
 The authors thank the referee for his (or her) careful reading and helpful comments.

\end{document}